\documentclass[a4paper,11pt,fleqn]{article}
\usepackage{pstricks}
\usepackage{amsmath}
\usepackage{amssymb}
\usepackage{mathrsfs} 
\usepackage{theorem} 
\usepackage{euscript}
\usepackage{exscale,relsize}
\usepackage{graphicx}
\usepackage{srcltx}
\usepackage{xcolor}
\usepackage{algorithm}
\usepackage{algorithmic}
\usepackage{charter}
\usepackage{url}
\newcommand{\email}[1]{\href{mailto:#1}{\nolinkurl{#1}}}
\renewcommand{\leq}{\ensuremath{\leqslant}}
\renewcommand{\geq}{\ensuremath{\geqslant}}
\newcommand{\minimize}[2]{\ensuremath{\underset{\substack{{#1}}}%
{\text{\rm minimize}}\;\;#2 }}

\newcommand{\argmind}[2]{\ensuremath{\underset{\substack{{#1}}}%
{\text{\rm argmin}}\;\;#2 }}
\newcommand{\argmax}[2]{\ensuremath{\underset{\substack{{#1}}}%
{\text{\rm argmax}}\;\;#2 }}

\newcommand{\Pair}[2]{\big\langle{{#1},{#2}}\big\rangle}

\newcommand{\menge}[2]{\big\{{#1}~\big |~{#2}\big\}}

\newcommand{\RX}{\ensuremath{\left(-\infty,+\infty\right]}}

\newcommand{\XX}{\ensuremath{{\mathscr X}}}
\newcommand{\HH}{\ensuremath{{\mathcal H}}}

\newcommand{\YY}{\ensuremath{{\mathcal Y}}}

\newcommand{\GG}{\ensuremath{{\mathcal G}}}

\newcommand{\prox}{\ensuremath{\text{\rm prox}}}

\newcommand{\RR}{\ensuremath{\mathbb{R}}}

\newcommand{\RPP}{\ensuremath{\left]0,+\infty\right[}}
\newcommand{\NN}{\ensuremath{\mathbb N}}
\newcommand{\KK}{\ensuremath{\mathbb K}}

\newcommand{\dom}{\ensuremath{\text{\rm dom}\,}}

\newcommand{\inte}{\ensuremath{\text{\rm int}\,}}

\newcommand{\ri}{\ensuremath{\text{\rm ri}\,}}

\newcommand{\norm}[1]{\|#1\|}

\newcommand{\BP}{\ensuremath{\EuScript P}}
\newcommand{\BD}{\ensuremath{\EuScript D}\,}
\newcommand{\BF}{\ensuremath{\EuScript F}}

\newtheorem{theorem}{Theorem}[section]
\newtheorem{lemma}[theorem]{Lemma}

\newtheorem{proposition}[theorem]{Proposition}

\theoremstyle{plain}{\theorembodyfont{\rmfamily}%
}
\theoremstyle{plain}{\theorembodyfont{\rmfamily}%
}
\theoremstyle{plain}{\theorembodyfont{\rmfamily}%
\newtheorem{remark}[theorem]{Remark}}
\theoremstyle{plain}{\theorembodyfont{\rmfamily}%
\theoremstyle{plain}{\theorembodyfont{\rmfamily}%
}
\theoremstyle{plain}{\theorembodyfont{\rmfamily}%
\newtheorem{definition}[theorem]{Definition}}
\theoremstyle{plain}{\theorembodyfont{\rmfamily}
\newtheorem{problem}[theorem]{Problem}}
\theoremstyle{plain}{\theorembodyfont{\rmfamily}
}

\numberwithin{equation}{section}
\setlength{\itemsep}{1pt} 
\usepackage[a4paper,colorlinks,breaklinks,unicode]{hyperref}

\begin{document}

\title{\sffamily\LARGE Smoothing technique for nonsmooth composite \\ minimization with linear operator}

\author{Quang Van Nguyen\footnotemark[1], \: Olivier Fercoq\footnotemark[2], \: and Volkan Cevher\footnotemark[1] \\[5mm]
\small
\small \footnotemark[1]\:  Laboratory for Information and Inference Systems (LIONS)\\
\small  \'Ecole Polytechnique F\'ed\'erale de Lausanne (EPFL), Switzerland\\
\small \footnotemark[2]\: T\'el\'ecom ParisTech, Institut Mines-T\'el\'ecom Paris, France\\
\small \email{quang.nguyen@epfl.ch},\: \email{olivier.fercoq@telecom-paristech.fr}\\
\small \email{volkan.cevher@epfl.ch}
\small\date{~}\\
}

\maketitle
\setcounter{page}{1}

\begin{abstract}
We introduce and analyze an algorithm for 
the minimization of convex functions that are the sum of
differentiable terms and proximable terms composed 
with linear operators. The method builds upon the recently developed smoothed gap technique.
In addition to a precise convergence rate result, valid even in the
presence of linear inclusion constraints, this new method allows an explicit treatment of the gradient of differentiable
functions and can be enhanced with line-search. We also 
study the consequences of restarting the acceleration of the algorithm at a
given frequency. These new features are not classical for 
primal-dual methods and allow us to solve difficult large scale 
convex optimization problems. We numerically illustrate the superior performance of the algorithm on basis pursuit,
TV-regularized least squares regression and L1 regression problems against the state-of-the-art.\end{abstract}



\section{Introduction}
Nonlinear and non-smooth convex optimization problems are widely presented in many disciplines, 
including signal and image processing, operations research, machine learning, game theory, 
economics, and mechanics. In this paper, we consider the following problem.
\begin{problem}
\label{p1}
Let $\HH$ and $\GG$ be real Hilbert spaces, let $M\colon\HH\to\GG$ be a bounded linear operator, 
and let  $f\colon\HH\to\RR$, $g\colon\HH\to\RX$ and $h\colon\GG\to\RX$ be proper, closed lower semi-continuous 
convex functions where $f$ is moreover assumed to have $L_f$-Lipschitz gradient. 
Consider the following generic convex minimization problem
\begin{equation}
\label{prob1}
F^\star = \min_{x\in\HH}\;\big\{f(x)+g(x)+h(Mx)\big\}
\end{equation}
under the assumption that its set of minimizers $\BP^\star$ is non-empty.
\end{problem}

Following \cite[Definition~19.11]{BC11_B}, if we suppose that $\emptyset\not=M(\dom (f+g))\cap\dom h = M(\dom g)\cap\dom h$ 
and set $\BF \colon\HH\times\GG\colon (x,y)\mapsto f(x)+g(x)+h(Mx -y)$ 
then Problem~\ref{p1} becomes the primal problem associated to $\BF$ and its associated dual problem is
\begin{equation}
\label{prob2}
G^\star = \max_{y\in\GG}\big\{G(y) = \min_{x\in\HH}\big\{f(x)+g(x)+\Pair{Mx}{y}-h^*(y)\big\}\big\},
\end{equation}
%
where $\dom{f}=\menge{x\in\HH}{f(x) < +\infty}$ is the domain of $f$ 
and $h^*\colon\GG\to\RX\colon y\mapsto\max\limits_{\bar y\in\GG}\;\Pair{\bar y}{y}-h(\bar y)$ 
is the Fenchel-Moreau conjugate function of $h$. In this case, 
\cite[Corollary~19.19]{BC11_B} states that the set of solution $\BD^\star$ to \eqref{prob2} is non-empty, 
and furthermore, a point $x^\star\in\HH$ is in $\BP^\star$ 
if and only if there exists $y^\star\in\BD^\star$ such that 
$(x^\star,y^\star)$ is a saddle point of the Lagrangian function
\begin{equation}
\mathcal{L}\colon (x,y)\mapsto f(x) + g(x) +\Pair{Mx}{y} - h^*(y).
\end{equation}
A particular case of Problem~\ref{p1} is when $h = \iota_{\KK}(\cdot - c)$ with $c\in\GG$, 
is the indicator function of  a non-empty closed convex subset $\KK\subset\GG$, i.e.,
\begin{equation}
\iota_{\KK}\colon\GG\to\RX\colon y\mapsto
\begin{cases}
0,&\text{if}\;y\in\KK,\\
+\infty,&\text{otherwise}.
\end{cases}
\end{equation}
In this case, Problem~\ref{p1} reduces to the following constrained minimization problem
\begin{equation}
\label{eq_constraint_nonsmooth}
\min_{x\in\HH}\big\{f(x)+g(x)\;:\; Mx - c\in\KK\big\},
\end{equation}
and furthermore, if $g = \iota_{\XX}$ for some $\XX\subset\HH$,
\begin{equation}
\min_{x\in\HH}\big\{f(x)\;:\; x\in\XX \;\text{such that}\; Mx - c\in\KK\big\}.
\end{equation}

A traditional approach for smooth minimization problems is the gradient descent algorithm together with its accelerated version. 
This idea has already been adapted for nonsmooth composite minimization problem by linearizing the smooth term before minimizing. 
For instance, if $h =0$, then Problem~\ref{p1} can be solved by FISTA (in other words, 
an accelerated forward-backward algorithm) \cite{BT09,CV13b}, 
and this approach can be generalized to the case where $h$ is with Lipschitz gradient. 
If furthermore, $h = \iota_{\{c\}}$ for some $c\in\GG$, then various of alternating direction optimization methods (ADMM) 
\cite{ADMM} can be used. A linearization technique is recently combined with ADMM in \cite{xu2016} to tackle such cases. 
However, in the general case, we need a special treatment of $h(Mx)$. 
For instance, we may compute approximations to the proximal operator of $(x \mapsto g(x) + h(Mx))$
as in~\cite{BT09tv}. We obtain an algorithm with a nested loop for this proximal operator computation. 
Provided we are able to control theaccuracy of the inner loop, we can obtain convergence rates.

Another possibility is to consider primal-dual splitting. 
By interpreting the optimization problem~\ref{p1} as a saddle point problem, 
we can derive methods updating primal and dual variables at each iteration, without
any nested loop. A primal-dual method able to deal with our composite framework was given in~\cite{condat13,Vu13}. 

A powerful smoothing framework was first introduced in \cite{Nes05a}, which can also be applied to solve Problem~\ref{p1}. 
The main idea isto consider a smoothed approximation to the nonsmooth function $h$
and minimize the resulting problem using an accelerated forward-backward algorithm.
This approach has been improved (for the case $f=0$) in \cite{VolkanQuocOlivier} as follows. 
Instead of considering a fixed smoothed approximation to the nonsmooth function $h$, 
the authors set up a homotopy strategy by considering
a decreasing smoothing parameter. In doing so, they obtain improved convergence characterizations, 
and,more importantly, they prove finite-time convergence rates in terms of
function value and infeasibility. Indeed, these rates are difficult to 
obtain in the constrained case when approximately solving the proximity operator
or considering classical primal-dual algorithms.

In this paper, we build on this latter, homotopy-based smoothing technique to 
tackle the more general composite framework, i.e., Problem~\ref{p1}. 
In this scenario, to apply the technique of \cite{Nes05a} as in \cite{VolkanQuocOlivier},
it would require the computation of the proximity operator of $f+g$ which is generally not easy 
even the case where one knows how to compute the proximity operators of $f$ and $g$ separately. 
One of our goals is to avoid this computational difficulty by using the smoothness.  
The second non-smooth part is then smoothed using the idea of \cite{Nes05a}. 
To see this, let us rewrite the objective function of \eqref{prob1} as follows
\begin{equation}
F(x) := f(x)+g(x)+h(Mx) = f(x)+g(x)+\max_{y\in\GG}\;\Pair{Mx}{y} - h^*(y),
\end{equation}
Instead of minimizing $F$, we first smooth one of its nonsmooth parts, says $h$, 
controlled by a smoothness parameter $\beta\in\RPP$ and then minimize
\begin{equation}
F_{\beta}(x) := f(x)+g(x)+\max_{y\in\GG}\;\big\{\Pair{Mx}{y} - h^*(y) -\beta q(y)\big\},
\end{equation}
with a suitable strongly convex function $q\colon\GG\to\RX$. 
We then use the accelerated forward-backward scheme 
to design algorithms that maintain the decrease of the approximated objective function in the sense that
\begin{equation}
(\forall k\in\NN)\quad F_{\beta_{k+1}}(\bar x^{k+1}) - F^\star \leq (1-\tau_k)\big(F_{\beta_k}(\bar x^k) - F^\star\big)+ \psi_k,
\end{equation}
where $(\bar x^k)_{k\in\NN}$ and the parameters are generated by the algorithm with 
$(\tau_k)_{k\in\NN}\subset [0, 1)^{\NN}$ and $(\max(\psi_k, 0))_{k\in\NN}$ tends to zero. 
We will also simultaneously update the $\beta_{k+1}$ parameter to zero to achieve an $\mathcal{O}(1/k)$ convergence rate. 
Our approach allows us to consider features that were introduced initially 
for unconstrained optimization like line-search or the balance of computational power between 
the steps of the algorithm.

The rest of the paper is organized as follow. In Section~\ref{sct:prem} we revise some technical facts. 
The main result is presented in Section~\ref{sct:main}. Numerical evidence is placed in Section~\ref{sct:num}. 

{\bf Notation.}
The Hilbert spaces $\HH$ and $\GG$ are equipped with their respective norms and inner products 
that we will both denote by $\|\cdot\|$ and $\Pair{\cdot}{\cdot}$ respectively. 
A positive definite linear operator $S$ on $\GG$, i.e., $\exists\sigma\in\RPP$ 
such that $(\forall y\in\GG)$ $\Pair{y}{Sy}\geq\sigma\|y\|^2$, induces a norm
$(\forall y\in\GG$ $\|y\|_S=\sqrt{\Pair{y}{Sy}}$. 
Given a proper, closed, lower semi-continuous convex function $f\colon\HH\to\RX$, 
we denote by $\inte\dom{f}$ the interior of $\dom{f}$ and by
\begin{equation}
\label{sub}
\partial{f}\colon x\mapsto\menge{v\in\HH}{(\forall y\in\HH)\; f(x)+\Pair{y-x}{v}\leq f(y)}
\end{equation} 
the subdifferential of $f$. If $f$ is differentiable, then we use $\nabla{f}$ for its gradient 
and in this case we say that $f$ has $L_f$-Lipschitz gradient with respect to norm $\|\cdot\|_S$ if 
\begin{equation}
(\forall x\in\HH)(\forall y\in\HH)\quad f(x)\leq f(y) +\Pair{x-y}{\nabla f(y)}+\frac{L_f}{2}\|x-y\|^2_S.
\end{equation}
Finally, we say that $f$ is $\mu$-strongly convex on $\HH$ with respect to $\|\cdot\|_S$ if
\begin{equation}
(\forall x\in\HH)(\forall y\in\HH)(\forall v\in\partial f(y))\quad f(x)\geq f(y)+\Pair{x-y}{v}+\frac{\mu}{2}\|x-y\|_S^2.
\end{equation}
Without indicating the norm, we are assuming that Lipschitz continuity or strong convexity is with 
a Hilbertian norm.

\section{Preliminaries}
\label{sct:prem}
In this section we revise some basic facts about the proximity operators and functions, and furthermore, 
the smoothing technique for non-smooth functions using the Fenchel-Moreau conjugate. 
In the optimization, the following notion of the proximity operators is widely used.
\begin{definition}
{\rm\cite[Definition~12.23]{BC11_B}}
Let $g\colon\HH\to\RX$ be a proper closed lower-semicontinuous convex function.  
The proximity operator of $g$ is
\begin{equation}
\prox_{g}\colon\HH\to\HH\colon x\mapsto\argmind{z\in\HH}{g(z)+\frac{1}{2}\|z-x\|^2}.
\end{equation}
\end{definition}

\begin{lemma}
{\rm\cite[Proposition~12.26]{BC11_B}}
\label{s:le2}
Let $g\colon\HH\to\RX$ be a proper closed lower semi-continuous convex function, 
let $\gamma\in\RPP$, let $x\in\HH$ and let $p=\prox_{\gamma g}(x)$. Then, it holds that
\begin{equation}
(\forall z\in\HH)\quad\gamma^{-1}\Pair{z-p}{x-p} + g(p)\leq g(z).
\end{equation}
\end{lemma}

The following smoothing technique using the Fenchel-Moreau conjugate and the proximity functions is from \cite{Nes05a}.

\begin{definition}
Let $h\colon\GG\to\RX$ be a convex function, let $\beta\in\RPP$, 
let $S$ be a positive definite linear operator on
$\GG$ and let $\dot y\in\GG$. The $\beta$-smooth approximation of $h$ is
\begin{equation}
\label{s:eq2}
h_{\beta}(\cdot; \dot y)\colon \GG\to\RX\colon y\mapsto\max\limits_{\bar y\in\GG}\big\{\Pair{y}{\bar y} - h^*(\bar y) - \frac{\beta}{2} \norm{\bar y - \dot y}^2_S \big\}.
\end{equation}
Set 
\begin{equation}
(\forall y\in\GG)\quad y_{\beta}^*(y;\dot y) =\argmax{{\bar y\in\GG}}{\Pair{y}{\bar y} - h^*(\bar y) -\frac{\beta}{2} \norm{\bar y-\dot y}^2_S }.
\end{equation} 
For instance, if $S = I$ then 
\begin{equation}
(\forall y\in\GG)\quad y_{\beta}^*(y;\dot y) = \prox_{\beta^{-1} h^*}\big(\beta^{-1}y + \dot y\big).
\end{equation}
\end{definition}

We summarize important properties of the smooth approximation in the following lemma 
which will be a crucial key in the analysis of our algorithm.
\begin{lemma}
\label{lem:move_beta}
Let $h\colon\GG\to\RX$ be convex, let $S$ be a positive definite linear operator on $\GG$ 
and let $\dot y\in\GG$. Consider the smooth approximations of $h$
\begin{equation}
(\forall \beta\in\RPP)\quad 
h_{\beta}(\cdot; \dot y)\colon \GG\to\RX\colon y\mapsto\max\limits_{\bar y\in\GG}\big\{\Pair{y}{\bar y} - h^*(\bar y) -\frac{\beta}{2} \norm{\bar y-\dot y}^2_S\big\}.
\end{equation}
Then the following hold:
\begin{enumerate}
\item\label{l1i}
Denote $\tilde h\colon (\beta, y)\mapsto h_{\beta}(y;\dot y)$. Then we have the following:
\begin{enumerate}
\item\label{lem:move_betai} 
$\tilde h$ is differentiable with respect to both variables and
\begin{multline}
(\forall y\in\GG)(\forall \beta\in\RPP)\quad\dfrac{\partial \tilde h}{\partial \beta}(\beta,y) = - \frac{1}{2}\norm{y_\beta^*(y;\dot y)-\dot y}^2_S 
 =  - \frac{1}{2}\norm{\nabla h_{\beta}(y;\dot y) - \dot y}^2_S.
\end{multline}
\item\label{lem:move_betaii} 
$\tilde h$ is convex with respect to first variable and 
\begin{equation}
\label{eq:convexity_h_beta}
\begin{aligned}
(\forall y\in\GG)(\forall \bar\beta\leq\tilde\beta)\quad
\tilde h(\bar\beta,y) &\leq \tilde h(\tilde\beta,y) - (\tilde\beta - \bar\beta) \frac{\partial\tilde h}{\partial \beta}(\bar\beta, y)\\
 &= \tilde h(\tilde\beta,y)  + \frac{\tilde\beta - \bar\beta}{2} \norm{\nabla h_{\bar\beta}(y;\dot y) -\dot y}^2_S.
 \end{aligned}
\end{equation}
\end{enumerate}
\item\label{lem:move_betaiii} 
Let $\beta\in\RPP$. Then the function $y\mapsto h_\beta(y;\dot y)$ is well-defined on $\GG$. It is convex 
with $\frac{1}{\beta}$-Lipschitz gradient in the norm $\norm{\cdot}_{S^{-1}}$ and furthermore,
\begin{equation}
\label{eq:cocoercivity_nabla_h_beta}
(\forall (\bar y, \hat y)\in\GG^2)\quad h_\beta(\hat y;\dot y) + \Pair{\bar y - \hat y}{\nabla h_\beta(\hat y;\dot y)} 
\leq h_\beta (\bar y;\dot y) - \frac{\beta}{2} \norm{\nabla h_\beta(\hat y;\dot y) - \nabla h_\beta(\bar y;\dot y)}^2_S.
\end{equation}
\item\label{lem:move_betaiv} 
The following inequality holds:
\begin{equation}
\label{eq:magic_return_to_zero}
(\forall (y, \hat y)\in\GG^2)(\forall\beta\in\RPP)\quad h_\beta(\hat y;\dot y) + \Pair{y - \hat y}{\nabla h_\beta(\hat y;\dot y)} 
\leq h(y) - \frac{\beta}{2} \norm{\nabla h_\beta(\hat y;\dot y) - \dot y}^2_S.
\end{equation}
\item\label{lem:move_betav} 
For all $(\beta,\tau)\in\RPP^2$ and for all $(\bar y, \hat y)\in\GG^2$, one has
\begin{multline}
\label{eq:combine_tau}
0\leq \norm{ (1- \tau) (\nabla h_\beta(\hat y;\dot y) - \nabla h_\beta(\bar y;\dot y)+ \tau (\nabla h_\beta(\hat y;\dot y) - \dot y)}^2_S  \\
= (1- \tau) \norm{ \nabla h_\beta(\hat y;\dot y) - \nabla h_\beta(\bar y;\dot y) }^2_S
+ \tau \norm{ \nabla h_\beta(\hat y;\dot y) - \dot y }^2_S - \tau (1-\tau) \norm{\nabla h_\beta(\bar y;\dot y) - \dot y}^2_S.
\end{multline}
\end{enumerate}
\end{lemma}

\begin{proof}
\ref{lem:move_betai} 
As there is a unique minimizer to the problem defining $h_\beta(\cdot;\dot y)$, the function
is differentiable with respect to $\beta$ and $y$.

\ref{lem:move_betaii} 
The function $\beta\mapsto h_\beta(y; \dot y)$ is convex as it is a maximum of functions, 
which are linear in $\beta$ indexed by $y$ and $\dot y$. 
The rest follows by convexity and the first point.

\ref{lem:move_betaiii} 
By the same arguments as in \cite[Theorem~1]{Nes05a} we deduce that the function 
$y\mapsto h_\beta(y;\dot y)$ is convex and finite on $\GG$. It also has $\frac{1}{\beta}$-Lipschitz gradient in the norm $\norm{\cdot}_{S^{-1}}$. 
\eqref{eq:cocoercivity_nabla_h_beta} is the cocoercivity inequality for convex functions with Lipschitz gradient. We provide the proof for completeness.
Define $\phi(z) = h_\beta(z; \dot y) - \langle \nabla h_\beta(\hat y; \dot y), z\rangle$. The function $\phi$ is convex, its minimum is attained at $\hat y$ and it has a $\frac{1}{\beta}$-Lipschitz gradient in the norm $\norm{\cdot}_{S^{-1}}$. Hence
\[
\phi(\hat y) \leq \phi(\bar y - \beta^{-1} S \nabla \phi(\bar y)) \leq \phi(\bar y) - \beta^{-1} \langle \nabla \phi(\bar y), S \nabla \phi(\bar y) \rangle+ \frac{\beta}{2} \norm{\beta^{-1} S \nabla \phi(\bar y)}^2_{S^{-1}}.
\]
We get the result because $\nabla \phi(\bar y) = \nabla h_\beta(\bar y; \dot y) -  \nabla h_\beta(\hat y; \dot y)$.

\ref{lem:move_betaiv} 
Let $(\bar y, \hat y)\in\GG^2$ and $\beta\in\RPP$ and set $y_{\beta}^* = y_{\beta}^*(\hat y;\dot y)= \nabla h_\beta(\hat y;\dot y)$. 
We have
\begin{equation}
\begin{aligned}
h_\beta(\hat y;\dot y) + \Pair{ y - \hat y}{\nabla h_\beta(\hat y;\dot y)}
&= \Pair{\hat y}{ y^*_\beta} - h^*(y^*_\beta) - \frac{\beta}{2} \norm{y^*_\beta-\dot y}^2_S + \Pair{y - \hat y}{y^*_\beta }\\
& = \Pair{y}{y^*_\beta} - h^*(y^*_\beta) - \frac{\beta}{2}\norm{y^*_\beta-\dot y}^2_S\\
& \leq \max_{\bar y \in \GG}\bigg\{ \Pair{y}{\bar y} - h^*(\bar y)\bigg\}- \frac{\beta}{2}\norm{y^*_\beta-\dot y}^2_S\\
& \leq h(y) - \frac{\beta}{2} \norm{\nabla h_\beta(\hat y;\dot y) -\dot y}^2_S.
\end{aligned}
\end{equation}

\ref{lem:move_betav} This follows from the classical equality
\begin{equation}
\norm{(1-\tau) a + \tau b}^2_S = (1-\tau) \norm{a}^2_S + \tau \norm{b}^2_S - \tau (1-\tau) \norm{a-b}^2_S.
\end{equation}
\end{proof}

\section{Main results}
\label{sct:main}

\subsection{Presentation of the algorithms}

In this section, we design new algorithms to solve Problem~\ref{p1} based on the smoothing technique introduced in the previous 
section. Consider the setting of Problem~\ref{p1}. We fix a positive definite linear operator $S$ on $\GG$ and a point $\dot y\in\GG$. 
We will need the operator norm of $M$ defined as
\[
\|M\|_{S^{-1}} = \sup_{x \not = 0} \frac{\norm{Mx}_{S^{-1}}}{\norm{x}}
\]

Our first algorithm is given as Algorithm~\ref{eq2}. 

\begin{algorithm} 
\begin{algorithmic}[1]
\STATE{Inputs: $\tau_0 = 1$, $\beta_0 > 0$, $\bar x^0\in\HH$, $\tilde x^0\in\HH$.}
\FOR{$k = 0, 1, \ldots$}
\STATE{$\hat x^k = (1-\tau_k) \bar x^k + \tau_k \tilde x^k$}
\STATE{$\beta_{k+1} = \frac{\beta_k}{1+\tau_k}$ and $B_{k+1}=L_f + \frac{\|M\|_{S^{-1}}^2}{\beta_{k+1}}$}
\STATE{$v^k =  M^*y_{\beta_{k+1}}^*(M\hat x^k;\dot y)$}\label{v_k}
\STATE{$\tilde x^{k+1} = \prox_{\frac{1}{\tau_kB_{k+1}}g}\big(\tilde x^k - \frac{1}{\tau_kB_{k+1}}(\nabla f(\hat x^k) + v^k)\big)$}\label{aug}
 \STATE{$\bar x^{k+1} = \hat x^k + \tau_k(  \tilde x^{k+1} - \tilde x^k) = (1-\tau_k) \bar x^k + \tau_k \tilde x^{k+1}$}
 \STATE{Find the unique positive $\tau_{k+1}$ such that}\label{tauk}
 \[
 \frac{B_{k+1} - L_f}{B_{k+1}}\tau_{k+1}^3 + \tau_{k+1}^2 +  \tau_{k}^2\tau_{k+1} -  \tau_{k}^2 = 0
 \]
\ENDFOR
\RETURN $\bar x^{k+1}$
\end{algorithmic}
\caption{Linearized ASGARD}
\label{eq2}
\end{algorithm}

\begin{remark}
Let us consider problem \eqref{eq_constraint_nonsmooth} with $\KK=\{0\}$. Recent paper \cite{xu2016} proposed a combination of linearizing technique and alternating direction of multipliers methods to solve \eqref{eq_constraint_nonsmooth} in which they obtained the $\mathcal{O}(1/k)$-rate (ergodic) convergence for fixed parameters and $\mathcal{O}(1/k^2)$ with adaptive parameters. Recall that in this case $h=\iota_{\{c\}}$ and hence $y^*_{\beta_{k+1}}(M\hat x^k)$ in the step~\ref{v_k} of Algorithm~\ref{eq2} (with $\dot y =0$ and $S=\text{I}$) becomes
\begin{equation}
y^*_{\beta_{k+1}}(M\hat x^k) = \prox_{\beta_{k+1}^{-1}h^*}(\beta_{k+1}^{-1}M\hat x^k) = \beta_{k+1}^{-1}\big(M\hat x^k - \prox_{\beta_{k+1}h}(M\hat x^k)\big)=\beta_{k+1}^{-1}\big(M\hat x^k - c\big),
\end{equation}
where we used Moreau's decomposition \cite[Theorem~14.3]{BC11_B}. Hence, for problem \eqref{eq_constraint_nonsmooth} with $\KK=\{0\}$, our Algorithm~\ref{eq2} is non-augmented Lagrangian version of \cite[Algorithm~1]{xu2016} where in Step~\ref{aug}, instead of using linearized augmented Lagrangian as \cite{xu2016}, we only used linearization of $f+g$ and hence this step only requires the computation of proximity of $g$. 
We also note that the update rule for parameters of \cite{xu2016} is different from ours.
\end{remark}

For some problems, for instance when $f$ encodes some data-fitting term, computing $\nabla f$ requires much more computational
power than $\prox_{h^*}$ or $\prox_g$. To circumvent this issue and concentrate 
on the non-smoothness of the objective, we propose
Algorithm~\ref{eq2b}, a variant of the standard ASGARD in which we use old gradients.
\begin{algorithm}[hptb]
\begin{algorithmic}[1]
\STATE{Inputs: $\tau_0 = 1$, $\beta_0 > 0$, $\bar x^0\in\HH$, $\tilde x^0\in\HH$, $\hat{\hat x}^0 \in \HH$, $\delta > 0$ and $\sigma > 0$.}
\FOR{$k = 0, 1, \ldots$}
\STATE{$\hat x^k = (1-\tau_k) \bar x^k + \tau_k \tilde x^k$}
\STATE{$\beta_{k+1} = \frac{\beta_k}{1+\tau_k}$ and $B_{k+1}=L_f + \frac{\|M\|_{S^{-1}}^2}{\beta_{k+1}}$}\label{Bb}
\STATE{$v^k =  M^*y_{\beta_{k+1}}^*(M\hat x^k;\dot y)$}
\STATE{$\tilde x^{k+1} = \prox_{\frac{1}{\tau_kB_{k+1}}g}\big(\tilde x^k - \frac{1}{\tau_kB_{k+1}}(\nabla f(\hat{\hat x}^k) + v^k)\big)$} \label{prox-step}
 \STATE{$\bar x^{k+1} = \hat x^k + \tau_k(  \tilde x^{k+1} - \tilde x^k) = (1-\tau_k) \bar x^k + \tau_k \tilde x^{k+1}$}
 \IF{$\frac 12\norm{\bar x^{k+1} - \hat{\hat x}^k}^2  - \frac 12\norm{\bar x^{k+1} - \hat x^k}^2
 \leq \sigma \big(\frac{\tau_k^2 B_{k+1}}{L_f} \big)^{2+\delta}$}
  \STATE $\hat{\hat x}^k = \hat{\hat x}^{k-1}$
 \ELSE
  \STATE $\hat{\hat x}^k = \hat x^k$
  \STATE Goto \ref{prox-step} and recompute $(\bar x^{k+1}, \tilde x^{k+1})$ with the true gradient.
 \ENDIF
 \STATE{Find the unique positive $\tau_{k+1}$ such that}
 \[
 \frac{B_{k+1} - L_f}{B_{k+1}}\tau_{k+1}^3 + \tau_{k+1}^2 +  \tau_{k}^2\tau_{k+1} -  \tau_{k}^2 = 0
 \]
\ENDFOR
\RETURN $\bar x^{k+1}$
\end{algorithmic}
\caption{Linearized ASGARD with old gradients}
\label{eq2b}
\end{algorithm}

\begin{remark}
Let us consider the case when $\hat{\hat x}^{k+1} = \hat x^k$. We have
\begin{align*}
&\frac 12\norm{\bar x^{k+2} - \hat{\hat x}^{k+1}}^2  - \frac 12\norm{\bar x^{k+2} - \hat x^{k+1}}^2
 = \langle \bar x^{k+2} - \frac{\hat{\hat x}^{k+1} + \hat x^{k+1}}{2}, \hat x^{k+1} - \hat{\hat{x}}^{k+1} \rangle \\
 & \qquad \qquad= \langle \tau_{k+1}(\tilde x^{k+2} - \tilde x^{k+1}) - \frac{\hat{\hat x}^{k+1} - \hat x^{k+1}}{2}, \hat x^{k+1} - \hat{\hat{x}}^{k+1} \rangle \\
& \hat x^{k+1} - \hat{\hat{x}}^{k+1} = \hat x^{k+1} - \hat x^{k}
 = \tau_{k+1} (\tilde x^{k+1} - \hat x^k) + \tau_k(1-\tau_{k+1})(\tilde x^{k+1} - \tilde x^k)
\end{align*}
and so assuming boundedness of the iterates, 
$\frac 12\norm{\bar x^{k+2} - \hat{\hat x}^{k+1}}^2  - \frac 12\norm{\bar x^{k+2} - \hat x^{k+1}}^2 $ is at most of the order of $\tau_k^2 \in \mathcal{O}(1/k^2)$. 
It is thus likely that it may sometimes be smaller than $\sigma \big(\frac{\tau_k^2 B_{k+1}}{L_f} \big)^{2+\delta}$ if $\delta$ is small enough.
\end{remark}

The last variant of our method is Algorithm~\ref{alg:asgardlinesearch}. It is equipped with a line search inspired by the line search 
for the accelerated universal gradient method~\cite{nesterov2015universal}.
It is particularly useful when the Lipschitz constant of $\nabla f$ is 
difficult to estimate. Moreover, it automatically adapts to the case when
$h$ is smooth by preventing the current estimate of the Lipschitz constant $B_{k+1}$
to increase to infinity. Note that because of the line search test on Step~\ref{step:linesearchtest}, the use of old gradients is not compatible with
this line search. Unlike the line search of Malitsky et al \cite{MP11}, the goal here is not only to
adaptively estimate $\|M\|^2$ but also the whole $L_f +\|M\|^2/\beta$. Our line search is more computationally demanding but may
take profit of some local smoothness of the nonsmooth function h.

\begin{algorithm}[hptb]
\begin{algorithmic}[1]
\STATE{Inputs: $\tau_0 = 1$, $\beta_0 > 0$, $\bar x^0\in\HH$, $\tilde x^0\in\HH$, $B_0 \leq a (L_f + \frac{\|M\|_{S^{-1}}^2 }{\beta_{0}})$, $a > 1$.}
\FOR{$k = 0,1, \ldots$}
\STATE{$B_{k+1} = a^{-1} B_{k}$}
\REPEAT
\STATE{$B_{k+1} = a B_{k+1}$}
\STATE{Find the unique positive $\tau_k$ such that $\frac{1 - \tau_k}{\tau_k^2 B_{k+1}} = \frac{1}{\tau_{k-1}^2 B_{k}}$  \quad (except $\tau_0=1$)} \label{step:tau}
\STATE{$\hat x^k = (1-\tau_k) \bar x^k + \tau_k \tilde x^k$}
\STATE{$\beta_{k+1} = \frac{\beta_k}{1 + \tau_k}$}
\STATE{$v^k =  M^*y_{\beta_{k+1}}^*(M\hat x^k;\dot y)$}
\STATE{$\tilde x^{k+1} = \prox_{\frac{1}{\tau_kB_{k+1}}g}\big(\tilde x^k - \frac{1}{\tau_kB_{k+1}}(\nabla f(\hat x^k) + v^k)\big)$}
 \STATE{$\bar x^{k+1} = \hat x^k + \tau_k(  \tilde x^{+} - \tilde x^k) = (1-\tau_k) \bar x^k + \tau_k \tilde x^{k+1}$}
\UNTIL{$f(\bar x^{k+1}) + h_{\beta_{k+1}}(M \bar x^{k+1}; \dot y) \leq f(\hat x^k) + h_{\beta_{k+1}}(M\hat x^k, \dot y) + \langle \nabla f(\hat x^k) + v^k, \bar x^{k+1} - \hat x^k\rangle + \frac{B_{k+1}}{2}\norm{\bar x^{k+1} - \hat x^k}^2$}
\label{step:linesearchtest}
\ENDFOR
\RETURN $\bar x^{k+1}$
\end{algorithmic}
\caption{Linearized ASGARD with line search}
\label{alg:asgardlinesearch}
\end{algorithm}

\subsection{Convergence of the parameters to 0}

To analyze the convergence of Algorithm~\ref{eq2}, we need the following result.

\begin{lemma} \label{lem:choose_parameters_asgard}
Let $(\tau_k)_{k\in\NN}$, $(\beta_k)_{k\in\NN}$ and $(B_{k+1})_{k\in\NN}$ 
be the positive sequences determined by Algorithm~\ref{eq2}. Then the following hold:
\begin{equation}
(\forall k\in\NN\backslash\{0\})\quad 
\frac{1-\tau_k}{\tau_k^2B_{k+1}} = \frac{1}{\tau_{k-1}^2B_k}.
\end{equation}
Furthermore,
\begin{equation}
\label{eq100a}
(\forall k\in\NN)\quad \frac{1}{k+1} \leq \tau_k \leq \frac{2}{k+2}, 
\quad \beta_k \leq \frac{\beta_0}{k+1}
\quad\text{and}\quad\tau_k^2B_{k+1}\leq\frac{\tau_0^2B_1}{k+1}.
\end{equation}
\end{lemma}

\begin{proof}
First it follows from Step~\ref{Bb} of Algorithms~\ref{eq2} and~\ref{eq2b} that
\begin{equation}
(\forall k\in\NN\backslash\{0\})\quad 
\frac{B_k+(B_k - L_f)\tau_k}{B_{k+1}} = \frac{L_f+ \frac{\|M\|_{S^{-1}}^2}{\beta_k} + \frac{\|M\|_{S^{-1}}^2}{\beta_k}\frac{(\beta_k-\beta_{k+1})}{\beta_{k+1}}}{L_f+ \frac{\|M\|_{S^{-1}}^2}{\beta_{k+1}}}
= \frac{L_f+ \frac{\|M\|_{S^{-1}}^2}{\beta_{k+1}}}{L_f+ \frac{\|M\|_{S^{-1}}^2}{\beta_{k+1}}} = 1,
\end{equation}
and Step~\ref{tauk} yields
\begin{equation}
(\forall k\in\NN\backslash\{0\})\quad 
\frac{1-\tau_k}{\tau_k^2B_{k+1}} = \frac{1}{\tau_{k-1}^2B_k}\frac{B_k + (B_k - L_f)\tau_k}{B_{k+1}} = \frac{1}{\tau_{k-1}^2B_k}.
\end{equation}
Next the definitions of $(\tau_k)_{k\in\NN}$ and $(\beta_k)_{k\in\NN}$ imply that
\begin{equation}
\label{eq1}
(\forall k\in\NN)\quad 
\alpha_{k+1}\tau_{k+1}^3 + \tau_{k+1}^2 +  \tau_{k}^2\tau_{k+1} -  \tau_{k}^2 = 0
\quad\text{and}\quad 
(1+\tau_{k})\beta_{k+1} =\beta_{k},
\end{equation}
where we define
\begin{equation}
(\forall  k\in\NN)\quad \alpha_{k+1} =  \frac{B_{k+1} - L_f}{B_{k+1}}\ \in\; [0, 1].
\end{equation}
For $\alpha \in \;[0, 1]$ and $\tau>0$, let us consider the following cubic function
\begin{equation}
(\forall t\in\RR_{+})\quad P(t)=\alpha t^3 + t^2 +  \tau^2 t -  \tau^2.
\end{equation}
On the one hand, since $(\forall t > 0)$ $P'(t) >0$ and $P(0) = -  \tau^2 < 0$ and $P(1) = \alpha + 1>0$,
we deduce that $P$ has a unique root $t_{+}\in \;]0, 1[$. On the other hand, 
let us define
\begin{equation}
(\forall t\in\RR_{+}) \quad Q(t) =  t^2 +  \tau^2t -  \tau^2.
\end{equation}
Then $(\forall t\in\RR_{+})$ $Q(t) \leq P(t)$, and in particular, $P(t^\star) \geq Q(t^\star) = 0 = P(t_{+})$, 
where $t^\star$ is the unique positive root of $Q$, that is $t^\star =(-\tau^2 + \sqrt{ \tau^4 + 4 \tau^2})/2$.  
As $P$ is nondecreasing on $\RPP$, we get $t_{+}\leq t^\star$. Consequently, since $\tau_0>0$, 
we deduce that $(\tau_k)_{k\in\NN}$ is well-defined and furthermore,
\begin{equation}
\label{eq101}
(\forall k\in\NN)\quad \tau_{k+1} \leq\frac{-\tau_{k}^2 + \sqrt{ \tau_{k}^4 + 4 \tau_{k}^2}}{2}.
\end{equation} 
This inequality and induction on $k\in\NN$ easily yields 
\begin{equation}
(\forall k\in\NN)\quad\tau_k \leq \frac{2}{k+2}.
\end{equation}
Now the two equalities in \eqref{eq1} imply that
\begin{equation}
\label{eq102}
(\forall k\in\NN)\quad
(\beta_k - \beta_{k+1} ) = \beta_{k+1} \tau_k\quad\text{and}\quad
\tau_{k}^2 = \tau_{k+1}^2\frac{1+\alpha_{k+1}\tau_{k+1}}{1-\tau_{k+1}} <  \tau_{k+1}^2\frac{1+\tau_{k+1}}{1-\tau_{k+1}}.
\end{equation}
We show by induction that
\begin{equation}
\label{siam:e1}
(\forall k\in\NN)\quad\tau_k \geq \frac{1}{k+1}.
\end{equation} 
Note that $\tau_0 = 1 \geq \frac{1}{0+1}$. Suppose that there exists $k_0\in\NN$ such that 
$\tau_{k_0} \geq \frac{1}{k_0+1}$ and that $\tau_{k_0+1} < \frac{1}{k_0+2}$. 
Then we deduce from \eqref{eq102} that
\begin{equation} 
\frac{1}{(k_0+1)^2} \leq \tau_{k_0}^2 < \tau_{k_0+1}^2 \frac{1 + \tau_{k_0+1}}{1-\tau_{k_0+1}}
< \frac{1}{(k_0+2)^2} \frac{1 + \frac{1}{k_0+2}}{1 - \frac{1}{k_0+2}} = \frac{1}{(k_0+2)^2} \frac{k_0 + 2}{k_0+1}
\end{equation}
which is equivalent to $(k_0+2)^2 < (k_0 +1)(k_0+2)$ which never happens. 
Hence, \eqref{siam:e1} holds true. We then deduce from induction that
\begin{equation}
(\forall k\in\NN)\quad\beta_{k+1} = \frac{\beta_k}{1 + \tau_k} \leq \beta_k \frac{k+1}{k+2}
\leq \beta_0 \prod_{l = 0}^k \frac{l+1}{l+ 2}=\dfrac{\beta_0}{k+2}.
\end{equation}
Of course $\beta_0 \leq \frac{\beta_0}{0+1}$ and hence,
\begin{equation}
(\forall k\in\NN)\quad\beta_k  \leq \dfrac{\beta_0}{k+1}.
\end{equation}
It again follows from induction and from \eqref{siam:e1} that
\begin{multline}
\label{eq205}
(\forall k\in\NN)\quad \tau_k^2B_{k+1}=(1-\tau_k)\tau_{k-1}B_k = \ldots = \prod_{l=1}^k(1-\tau_l)\tau_0^2B_1\\
\leq \prod_{l=1}^k\bigg(1-\frac{1}{l+1}\bigg)\tau_0^2B_1
= \tau_0^2B_1\prod_{l=1}^k\frac{l}{l+1} =\frac{B_1}{k+1}.
\end{multline}
\end{proof}

The convergence of Algorithm~\ref{alg:asgardlinesearch} is based on the following asymptotic property of parameters.

\begin{lemma} \label{lem:choose_parameters_asgard2}
Let $(\tau_k)_{k\in\NN}$, $(\beta_k)_{k\in\NN}$ and $(B_{k+1})_{k\in\NN}$ be the positive sequences 
determined by Algorithm~\ref{alg:asgardlinesearch}. Then $(\forall k\in\NN)$ $\tau_{k} \leq \frac{2}{k+2}$ and 
\begin{equation}
\begin{aligned}
(\forall k\in\NN)\quad\frac{\beta_0}{2B_1}\tau_k^2B_{k+1}&\leq\beta_{k+1} \leq \frac 12 \Big(3 a \frac{\beta_0}{B_1} \tau_k^2 L_f + \sqrt{(3 a \frac{\beta_0}{B_1} \tau_k^2 L_f)^2 + 12 a \frac{\beta_0}{B_1} \tau_k^2 \|M\|_{S^{-1}}^2} \Big) \\
&\sim \sqrt{3a \frac{\beta_0}{B_1}}\|M\|_{S^{-1}}\tau_k = \mathcal{O}\bigg(\frac{1}{k}\bigg).
\end{aligned}
\end{equation}
\end{lemma}

\begin{proof}
For every $k\in\NN$, since Lemma~\ref{lem:move_beta}~\ref{lem:move_betaiii} states that the function $y\mapsto h_{\beta_{k+1}}(y;\dot y)$ has $\frac{1}{\beta_{k+1}}$-Lipschitz gradient, we deduce that the function $x \mapsto f(x)+ h_{\beta_{k+1}}(Mx;\dot y)$ has $L_f + \frac{\|M\|_{S^{-1}}^2}{\beta_{k+1}}$-Lipschitz gradient and thus, in Algorithm~\ref{alg:asgardlinesearch}, when
the line search terminates, one necessarily has $B_{k+1} \leq a (L_f + \frac{\|M\|_{S^{-1}}^2}{\beta_{k+1}})$. 
Now we deduce from the definition of $(\tau_k)_{k\in\NN}$ that 
\begin{equation}
(\forall k\in\NN)\quad \tau_k^2 B_{k+1}=(1-\tau_k)\tau_{k-1}^2B_k=\tau_0^2B_1\prod_{l=1}^k(1-\tau_l),
\end{equation}
so we need to study the bounds for $(\tau_k)_{k\in\NN}$. We now prove by induction that 
\begin{equation}
\label{eq200} 
(\forall k\in\NN)\quad\tau_k \leq \frac{2}{k+2}.
\end{equation}
We clearly have $\tau_0 = 1 \leq \frac{2}{0+2}$. Suppose that there exists $k_0\in\NN\backslash\{0\}$ 
such that $\tau_{k_0-1} \leq \frac{2}{k_0+1}$ and that $\tau_{k_0} > \frac{2}{k_0+2}$. 
Then, as $B_{k_0+1} \geq B_{k_0}$, we have $\frac{1-\tau_{k_0}}{\tau_{k_0}^2} = \frac{B_{k_0+1}}{\tau_{k_0-1}^2B_{k_0}} \geq \frac{1}{\tau_{k_0-1}^2}$. This would lead to 
\begin{equation}
\frac{k_0^2 + 2k_0 +1}{4} = \frac{(k_0+1)^2}{4} \leq \frac{1}{\tau_{k_0-1}^2}
\leq \frac{1-\tau_{k_0}}{\tau_{k_0}^2} < \frac{(k_0+2)^2}{4} - \frac{k_0+2}{2} = \frac{k_0^2 + 2 k_0}{4}.
\end{equation}
This contradiction proves \eqref{eq200}. Since we have
\begin{equation}
(\forall k\in\NN)\quad\beta_{k+1} = \beta_1 \prod_{l=1}^k \frac{1}{1+\tau_l} = \beta_1 \frac{\prod_{l=1}^k(1-\tau_l)}{\prod_{l=1}^k (1 - \tau_l^2)}= \frac{\beta_0}{2} \frac{\prod_{l=1}^k(1-\tau_l)}{\prod_{l=1}^k (1 - \tau_l^2)},
\end{equation}
and since
\begin{equation}
(\forall k\in\NN)\quad 1 \geq \prod_{l=1}^k (1 - \tau_l^2) \geq \prod_{l=1}^k \Big(1 - \frac{4}{(l+2)^2}\Big) = \frac{(k+4)(k+3) 2 . 1 }{(k+2)(k+1) 4 . 3} \geq \frac 16,
\end{equation}
we get
\begin{equation}
(\forall k\in\NN)\quad\frac{\beta_0}{2} \prod_{l=1}^k(1-\tau_l) \leq \beta_{k+1} \leq 3 \beta_0 \prod_{l=1}^k(1-\tau_l).
\end{equation}
Therefore,
\begin{equation}
(\forall k\in\NN)\quad\frac{\beta_0}{2} \frac{\tau_k^2 B_{k+1}}{B_1} \leq \beta_{k+1} \leq 3 \beta_0 \frac{\tau_k^2 B_{k+1}}{B_1}
\end{equation}
Since $(\forall k\in\NN)$ $B_{k+1} \leq a (L_f + \frac{\|M\|_{S^{-1}}^2}{\beta_{k+1}})$, it follows that 
\begin{equation}
(\forall k\in\NN)\quad\beta_{k+1}^2 - 3a\frac{\beta_0}{B_1}\tau_k^2\bigg[L_f \beta_{k+1} + \|M\|_{S^{-1}}^2\bigg] \leq 0,
\end{equation}
which implies that 
\begin{equation}
(\forall k\in\NN)\quad\beta_{k+1} \leq \frac 12 \Bigg(3 a \frac{\beta_0}{B_1} \tau_k^2 L_f + \sqrt{\bigg[3 a \frac{\beta_0}{B_1} \tau_k^2 L_f\bigg]^2 + 
12 a \frac{\beta_0}{B_1} \tau_k^2 \|M\|_{S^{-1}}^2} \Bigg)
 \sim \sqrt{3a \frac{\beta_0}{B_1}}\|M\|_{S^{-1}}\tau_k = \mathcal{O}\bigg(\frac{1}{k}\bigg)
\end{equation}
and hence,
\begin{equation}
(\forall k\in\NN)\quad\tau_k^2 B_{k+1} \leq \frac{2B_1}{\beta_0}\beta_{k+1} = \mathcal{O}\bigg(\frac{1}{k}\bigg).
\end{equation}
\end{proof}

\subsection{Speed of convergence}

The convergence theorem is based on the decrease of the smoothed optimality gap.
We prove in Proposition~\ref{t1} that for every iteration $k\in\NN$, $F_{\beta_{k+1}} - F^\star$ decreases 
as $\mathcal{O}(1/k)$. Then, using \cite[Lemma 2.1]{VolkanQuocOlivier} and the decrease of
the smoothness parameter to 0, we get the speed of convergence in 
function value and infeasibility.

\begin{proposition}
\label{t1}
Consider the setting of Problem~\ref{p1}. Let $(\bar x^k)_{k\in\NN}$ 
be generated by the ASGARD variants (Algorithms~\ref{eq2}, \ref{eq2b}, \ref{alg:asgardlinesearch})  and define 
\begin{equation}
(\forall k\in\NN)\quad F_{\beta_k}\colon x\mapsto f(x)+ g(x) + h_{\beta_k}(Mx; \dot y).
\end{equation} 
Then for  $x^\star\in\BP^\star$, we have:
\begin{equation}
(\forall k\in\NN)\quad F_{\beta_{k+1}}(\bar x^{k+1}) - F^\star \leq 
\begin{cases}
\frac{B_1}{2(k+1)}\|x^\star-\tilde x^0\|^2, &\text{for Algorithm~\ref{eq2}},\\
\frac{B_1}{k+1} \Big(\frac 12 \|x^\star-\tilde x^0\|^2 + \sigma \left(\frac{B_1}{2L_f}\right)^{1+\delta} \zeta(1+\delta) \Big), &\text{for Algorithm~\ref{eq2b}},\\
\frac{ B_1\beta_{k+1}}{\beta_0}\|x^\star-\tilde x^0\|^2 = \mathcal{O}\bigg(\frac{\|x^\star - \tilde x_0\|^2}{k}\bigg), &\text{for Algorithm~\ref{alg:asgardlinesearch}},
\end{cases}
\end{equation}
where $\zeta(s) = \sum_{i=0}^{+\infty} \frac{1}{(i+1)^s}$.
\end{proposition}

\begin{proof}
First we note that the arguments for Algorithms~\ref{eq2} and ~\ref{alg:asgardlinesearch} 
are similar to those of Algorithm~\ref{eq2b} by setting $(\forall k\in\NN)$ $\hat{\hat x}^k = \hat x^k$.  
We therefore prove the convergence for Algorithm~\ref{eq2b}. Now let us fix $x^\star\in\BP^\star$. Then, we have
\begin{multline}
\label{s:e4}
(\forall k\in\NN)\quad f(\bar x^{k+1}) + h_{\beta_{k+1}}(M\bar x^{k+1};\dot y) \leq f(\hat{\hat x}^k) + h_{\beta_{k+1}}(M\hat x^k;\dot y)
+ \Pair{\bar x^{k+1} - \hat{\hat x}^k}{\nabla f(\hat{\hat x}^k)} \\
+ \Pair{\bar x^{k+1} - \hat x^k}{v^k)}+ \frac{L_f}{2}\norm{\bar x^{k+1} - \hat{\hat x}^k}^2 
+ \frac{\|M\|_{S^{-1}}^2}{2\beta_{k+1}} \norm{\bar x^{k+1} - \hat x^k}^2.
\end{multline}
Because $g$ is convex and because Algorithm~\ref{eq2b} yield $(\forall k\in\NN)$ $\bar x^{k+1} = (1-\tau_k)\bar x^k + \tau_k\tilde x^{k+1}$, we get
\begin{equation}
\label{s:e5}
(\forall k\in\NN)\quad g(\bar x^{k+1})\leq (1-\tau_k) g( \bar x^k)+ \tau_k g(\tilde x^{k+1}).
\end{equation}
It now follows from Lemma~\ref{s:le2} that
\begin{equation}
\label{s:e6}
\begin{aligned}
(\forall k\in\NN)\quad g(\tilde x^{k+1}) &\leq g(x^\star) -\tau_kB_{k+1}\Pair{x^\star - \tilde x^{k+1}}{\tilde x^k -\tilde x^{k+1} -\tau_k^{-1}B_{k+1}^{-1}\big(\nabla f(\hat{\hat x}^k)+v^k\big)}\\
&= g(x^\star) -\tau_kB_{k+1}\Pair{x^\star - \tilde x^{k+1}}{\tilde x^k-\tilde x^{k+1} } +\Pair{x^\star - \tilde x^{k+1}}{\nabla f(\hat{\hat x}^k)+v^k}\\
&= g(x^\star) +\Pair{x^\star - \tilde x^{k+1}}{\nabla f(\hat{\hat x}^k)+v^k}\\
&+\frac{\tau_kB_{k+1}}{2}\|x^\star -\tilde x^k\|^2 - \frac{\tau_kB_{k+1}}{2}\|x^\star -\tilde x^{k+1}\|^2 - \frac{\tau_kB_{k+1}}{2}\|\tilde x^{k+1} -\tilde x^k\|^2.
\end{aligned}
\end{equation}
In turn, we obtain from \eqref{s:e4} and the fact that $(\forall k\in\NN)$ $\bar x^{k+1} - \hat{\hat x}^k = \hat x^k - \hat{\hat x}^k + \tau_k(\tilde x^{k+1} - \tilde x^k)$ that 
\begin{equation}
\label{s:e10}
\begin{aligned}
(\forall k\in\NN)\quad F_{\beta_{k+1}}(\bar x^{k+1}) & \leq f(\hat{\hat x}^k) + h_{\beta_{k+1}}(M\hat x^k;\dot y) + (1-\tau_k)g(\bar x^k) +\tau_kg(x^\star) 
+\tau_k\Pair{x^\star -\tilde x^k}{v^k}\\
&+\Pair{\hat x^k - \hat{\hat x}^k+ \tau_k(x^\star -\tilde x^k)}{\nabla f(\hat{\hat x}^k)} +\frac{\tau_k^2B_{k+1}}{2}\|x^\star -\tilde x^k\|^2 \\
&- \frac{\tau_k^2B_{k+1}}{2}\|x^\star -\tilde x^{k+1}\|^2
+\frac{L_f}{2}\norm{\bar x^{k+1} - \hat{\hat x}^k}^2  - \frac{L_f}{2}\norm{\bar x^{k+1} - \hat x^k}^2.
\end{aligned}
\end{equation}
Since Algorithm~\ref{eq2b} yield $(\forall k\in\NN)$ $\tau_k (\hat x^k - \tilde x^k) = (1-\tau_k)(\bar x^k - \hat x^k)$, we have
\begin{equation}
(\forall k\in\NN)\quad\tau_k (x^\star  - \tilde x^k) = \tau_k (x^\star  - \hat x^k) +  (1-\tau_k)(\bar x^k - \hat x^k),
\end{equation}
and hence,
\begin{equation}
(\forall k\in\NN)\quad\tau_k\Pair{x^\star -\tilde x^k}{v^k} = \tau_k\Pair{x^\star -\hat x^k}{v^k}
+(1-\tau_k)\Pair{\bar x^k -\hat x^k}{v^k}.
\end{equation}
Moreover, since $(\forall k\in\NN)$ $\hat x^k - \hat{\hat x}^k+ \tau_k(x^\star -\tilde x^k) = 
 \tau_k(x^\star - \hat{\hat x}^k) + (1-\tau_k)(\bar x^k - \hat{\hat x}^k)$, we get
\begin{equation}
(\forall k\in\NN)\quad \Pair{\hat x^k - \hat{\hat x}^k+ \tau_k(x^\star -\tilde x^k)}{\nabla f(\hat{\hat x}^k)} =  \tau_k\Pair{x^\star -\hat{\hat x}^k}{\nabla f(\hat{\hat x}^k)}
+(1-\tau_k)\Pair{\bar x^k -\hat{\hat x}^k}{\nabla f(\hat{\hat x}^k)}.
\end{equation}
Because $f$ is convex, \eqref{sub} yields
\begin{equation}
(\forall k\in\NN)\quad f(\hat{\hat x}^k)+\Pair{\bar x^k-\hat{\hat x}^k}{\nabla f(\hat{\hat x}^k)}\leq f(\bar x^k)
\quad\text{and}\quad 
f(\hat{\hat x}^k)+\Pair{x^\star-\hat{\hat x}^k}{\nabla f(\hat{\hat x}^k)}\leq f(x^\star),
\end{equation}
and hence, we derive from \eqref{eq:cocoercivity_nabla_h_beta} that
\begin{multline}
\label{s:e7}
(\forall k\in\NN)\quad f(\hat{\hat x}^k) + h_{\beta_{k+1}}(M\hat x^k; \dot y)+  \Pair{\bar x^k  - \hat{\hat x}^k}{\nabla f(\hat{\hat x}^k) } + \Pair{\bar x^k  - \hat x^k}{ v^k}\\
\leq  f(\bar x^k) + h_{\beta_{k+1}}(M\bar x^k; \dot y) - \frac{\beta_{k+1}}{2}\norm{\nabla h_{\beta_{k+1}}(M\hat x^k; \dot y) -\nabla h_{\beta_{k+1}}(M\bar x^k; \dot y)}_S^2.
\end{multline}
and from \eqref{eq:magic_return_to_zero} that
\begin{multline}
\label{s:e8}
(\forall k\in\NN)\quad f(\hat{\hat x}^k) + h_{\beta_{k+1}}(M\hat x^k; \dot y)+  \Pair{x^\star  - \hat{\hat x}^k}{\nabla f(\hat{\hat x}^k)}+  \Pair{x^\star  - \hat x^k}{ v^k}\\
\leq  f(x^\star) + h(Mx^\star) -  \frac{\beta_{k+1}}{2}\norm{\nabla h_{\beta_{k+1}}(M\hat x^k; \dot y) - \dot y}_S^2.
\end{multline}
On the other hand, we deduce from \eqref{eq:convexity_h_beta} that 
\begin{equation}
\label{s:e9}
\begin{aligned}
(\forall k\in\NN)\quad h_{\beta_{k+1}}(M\bar x^k; \dot y)
&\leq h_{\beta_{k}}(M\bar x^k; \dot y) + \frac{\beta_k - \beta_{k+1}}{2}\norm{\nabla h_{\beta_{k+1}}(M\bar x^k; \dot y) - \dot y}_S^2
\end{aligned}
\end{equation}
and from \eqref{eq:combine_tau} that
\begin{multline}
(\forall k\in\NN)\quad - (1-\tau_k)\frac{\beta_{k+1}}{2}\norm{\nabla h_{\beta_{k+1}}(M\hat x^k; \dot y) -\nabla h_{\beta_{k+1}}(M\bar x^k; \dot y)}_S^2 - \tau_k \frac{\beta_{k+1}}{2}\norm{\nabla h_{\beta_{k+1}}(M\hat x^k; \dot y) - \dot y}_S^2 \\
\leq -\tau_k(1-\tau_k)\frac{\beta_{k+1}}{2}\norm{\nabla h_{\beta_{k+1}}(M\bar x^k; \dot y) - \dot y}_S^2
\end{multline}
Altogether, by combining \eqref{s:e10} and \eqref{s:e7}-\eqref{s:e9}, we get
\begin{equation}
\begin{aligned}
(\forall k\in\NN)\quad F_{\beta_{k+1}}(\bar x^{k+1})& \leq (1-\tau_k) \big(g( \bar x^k) + f(\hat{\hat x}^k) + h_{\beta_{k+1}}(M \hat x^k;\dot y) + \Pair{\bar x^k - \hat{\hat x}^k}{\nabla f(\hat{\hat x}^k)}+ \Pair{\bar x^k - \hat x^k}{ v^k}\big)   \\
&\quad+ \tau_k \big( g(x^\star ) +f(\hat x^k) + h_{\beta_{k+1}}(M \hat x^k;\dot y) +  \Pair{ x^\star  - \hat{\hat x}^k}{\nabla f(\hat{\hat x}^k)}+\Pair{ x^\star  - \hat x^k}{v^k}\big)  \\
& \hskip - 9mm+ \tau_k^2 \frac{B_{k+1}}{2}\norm{x^\star - \tilde x^k}^2- \tau_k^2 \frac{B_{k+1}}{2}\norm{x^\star  - \tilde x^{k+1}}^2 +\frac{L_f}{2}\norm{\bar x^{k+1} - \hat{\hat x}^k}^2  - \frac{L_f}{2}\norm{\bar x^{k+1} - \hat x^k}^2  \\
&\leq(1-\tau_k) \big(g( \bar x^k) + f(\bar x^k) + h_{\beta_{k}}(M \bar x^k;\dot y)\big)+ \tau_k \big( g(x^\star) +f(x^\star) + h(Mx^\star) \big)   \\
& \qquad + \Big[\frac{(1-\tau_k)(\beta_k - \beta_{k+1})}{2} - \frac{\tau_k (1-\tau_k)\beta_{k+1}}{2} \Big]\norm{\nabla h_{\beta_{k+1}}(M \bar x^k;\dot y) - \dot y}_S^2   \\
& \hskip -9 mm+ \tau_k^2 \frac{B_{k+1}}{2}\norm{x^\star- \tilde x^k}^2- \tau_k^2 \frac{B_{k+1}}{2}\norm{x^\star  - \tilde x^{k+1}}^2
+\frac{L_f}{2}\norm{\bar x^{k+1} - \hat{\hat x}^k}^2  - \frac{L_f}{2}\norm{\bar x^{k+1} - \hat x^k}^2.
\end{aligned}
\end{equation}
It follows from the definition of $(\tau_k)_{k \geq 0}$ and $(\beta_k)_{k \geq 0}$ that 
\begin{multline}
\label{eq201}
(\forall k\in\NN)\quad F_{\beta_{k+1}}(\bar x^{k+1}) - F^\star + \frac{\tau_{k}^2 B_{k+1}}{2}\norm{x^\star - \tilde x^{k+1}}^2
\leq (1-\tau_k) (F_{\beta_k}(\bar x^k)- F^\star)+ \frac{\tau_{k}^2 B_{k+1}}{2}\norm{x^\star - \tilde x^k}^2 \\
+\frac{L_f}{2}\norm{\bar x^{k+1} - \hat{\hat x}^k}^2  - \frac{L_f}{2}\norm{\bar x^{k+1} - \hat x^k}^2 .
\end{multline}
On the one hand, \eqref{eq201} yields
\begin{equation}
\label{eq202}
F_{\beta_1}(\bar x^1)- F^\star\leq \frac{\tau_0^2 B_1}{2}\|x^\star-\tilde x^0\|^2
+ \tau_0^2B_1\sum_{i=0}^0 \frac{L_f}{2 \tau_i^2 B_{i+1}}\big(\norm{\bar x^{i+1} - \hat{\hat x}^i}^2  - \norm{\bar x^{i+1} - \hat x^i}^2 \big).
\end{equation}
On the other hand, since
\begin{equation}
(\forall k\in\NN\backslash\{0\})\quad 
\frac{1-\tau_k}{\tau_k^2B_{k+1}}=\frac{1}{\tau_{k-1}^2B_k},
\end{equation}
it follows from \eqref{eq201} and \eqref{eq202} that
\begin{equation}
\label{eq203}
\begin{aligned}
(\forall k\in\NN\backslash\{0\})\quad\frac{1}{\tau_{k}^2 B_{k+1}}&(F_{\beta_{k+1}}(\bar x^{k+1})- F^\star) + \frac{1}{2}\norm{x^\star - \tilde x^{k+1}}^2 \\
&\hskip -15 mm\leq \frac{1-\tau_k}{\tau_{k}^2 B_{k+1}} (F_{\beta_{k}}(\bar x^k) - F^\star)+ \frac{1}{2}\norm{x^\star - \tilde x^k}^2 +\frac{L_f}{2 \tau_k^2 B_{k+1}}\big(\norm{\bar x^{k+1} - \hat{\hat x}^k}^2  - \norm{\bar x^{k+1} - \hat x^k}^2 \big)\\
 &\hskip -15 mm = \frac{1}{\tau_{k-1}^2 B_{k}} (F_{\beta_{k}}(\bar x^k) - F^\star)+ \frac{1}{2}\norm{x^\star - \tilde x^k}^2+\frac{L_f}{2 \tau_k^2 B_{k+1}}\big(\norm{\bar x^{k+1} - \hat{\hat x}^k}^2  - \norm{\bar x^{k+1} - \hat x^k}^2 \big)\\
 & \hskip -15 mm \leq\frac{1}{\tau_0^2 B_1} (F_{\beta_{1}}(\bar x^1)- F^\star)+ \frac{1}{2}\norm{x^\star - \tilde x^1}^2 + \sum_{i=1}^k \frac{L_f}{2 \tau_i^2 B_{i+1}}\big(\norm{\bar x^{i+1} - \hat{\hat x}^i}^2  - \norm{\bar x^{i+1} - \hat x^i}^2 \big)\\
 & \hskip -15 mm \leq\frac{1}{2}\norm{x^\star - \tilde x^0}^2 + \sum_{i=0}^k \frac{L_f}{2 \tau_i^2 B_{i+1}}\big(\norm{\bar x^{i+1} - \hat{\hat x}^i}^2  - \norm{\bar x^{i+1} - \hat x^i}^2 \big).
\end{aligned}
\end{equation}
Altogether, \eqref{eq202} and \eqref{eq203} yield
\begin{equation}
\label{eq:decrease_smoothed_gap}
(\forall k\in\NN)\quad F_{\beta_{k+1}}(\bar x^{k+1})- F^\star\leq \frac{\tau_k^2 B_{k+1}}{2}\|x^\star-\tilde x^0\|^2
+ \tau_k^2B_{k+1}\sum_{i=0}^k \frac{L_f}{2 \tau_i^2 B_{i+1}}\big(\norm{\bar x^{i+1} - \hat{\hat x}^i}^2  - \norm{\bar x^{i+1} - \hat x^i}^2 \big).
\end{equation}
We note that in the above inequalities when $(\forall k\in\NN)$ $\hat{\hat x}^k = \hat x^k$ then the second term in the right hand side of the last line vanishes. 
Otherwise, the test 
\begin{equation}
\frac 12\norm{\bar x^{i+1} - \hat{\hat x}^i}^2  - \frac 12\norm{\bar x^{i+1} - \hat x^i}^2
\leq \sigma \big(\frac{\tau_i^2 B_{i+1}}{L_f} \big)^{2+\delta}
\end{equation}
ensures that the additional sum is uniformly bounded since $(\forall k\in\NN)$ $\tau_k^2 B_{k+1} \in \mathcal{O}(1/k)$. 
To conclude, we combine this estimate with Lemma~\ref{lem:choose_parameters_asgard} 
and Lemma~\ref{lem:choose_parameters_asgard2}.
\end{proof}

We are now ready to state the convergence result for Algorithm~\ref{eq2}. 
The convergence characterizations for the other variants of ASGARD follow 
mutatis mutandis for the other variants of ASGARD using 
the same arguments.  As in~\cite{VolkanQuocOlivier}, we consider two important particular cases: 
the case of equality constraints ($h = \iota_{\{c\}}$) and the
case where $h$ is Lipschitz continuous.
\begin{theorem}
\label{t3}
Let $(\bar x^k)_{k\in\NN}$  be the sequence generated by Algorithm~\ref{eq2}.
\begin{enumerate}
\item\label{t3:i}
Suppose that $h = \iota_{\{c\}}$ for some $c \in \GG$ and that 
$\ri(\dom{g})\cap\{x\in\HH\;|\; Mx = c\}\not=\emptyset$. Denote $F\colon x\mapsto f(x) + g(x)$. 
Then the following bounds hold for all $k\in\NN$ and $(x^\star, y^\star)\in\BP^\star\times\BD^\star$:
\begin{equation}
\begin{aligned}
F(\bar x^{k+1}) - F(x^{\star}) & \geq  - \norm{y^\star}_S \norm{M \bar x^{k+1} - c}_{S^{-1}} \\
F(\bar x^{k+1}) - F(x^{\star}) &\leq \frac{1}{k+1}\bigg(\frac{L_f}{2}+\frac{\|M\|_{S^{-1}}^2}{2\beta_0}\bigg) \norm{\tilde{x}^0 - x^\star}^2 + \norm{y^\star}_S \norm{M \bar{x}^{k+1} -c}_{S^{-1}} +\frac{\beta_{0}}{2(k+1)} \norm{y^\star - \dot{y}}_S^2 \\
 \norm{M \bar{x}^{k+1} - c}_{S^{-1}}  &\leq  \frac{\beta_0}{k+1}\bigg[ \norm{y^\star - \dot{y}}_S  + \bigg(\norm{y^\star - \dot{y}}_S^2 + \frac{2}{\beta_0}\bigg(\frac{L_f}{2}+\frac{\|M\|_{S^{-1}}^2}{2\beta_0}\bigg) \norm{\tilde{x}^0 - x^\star}^2 \bigg)^{1/2}\bigg]. 
\end{aligned}
\end{equation}

\item\label{t3:ii}
Suppose that $h$ is $D_\YY$-Lipschitz continuous in the $S$-norm, and denote $F(x) = f(x) + g(x) + h(Mx)$. Then the following bound holds
\begin{equation*}
(\forall k\in\NN)\quad F(\bar x^{k+1}) - F^\star \leq\frac{1}{k+1}\bigg(\frac{L_f}{2} + \frac{\|M\|_{S^{-1}}^2}{2\beta_0}\bigg)\norm{\tilde x^0 - x^\star}^2 
+ \frac{\beta_0}{k+1}\big[D_{\YY}^2 + \|\dot y\|^2\big] .
\end{equation*}
\end{enumerate}
\end{theorem}
\begin{proof}
\ref{t3:i} Fix $k\in\NN$. 
Using \cite[Lemma 2.1]{VolkanQuocOlivier}, we get
\begin{align*}
F(\bar x^{k+1}) - F^{\star} & \geq  - \norm{y^\star}_S \norm{M \bar x^{k+1} - c}_{S^{-1}} \\
F(\bar x^{k+1}) - F^{\star} &\leq F_{\beta_{k+1}}(\bar x^{k+1}) - F^{\star} + \norm{y^\star}_S \norm{M \bar{x}^{k+1} -c}_{S^{-1}} + \frac{\beta_{k+1}}{2} \norm{y^\star - \dot{y}}_S^2\\
 \norm{M \bar{x}^{k+1} - c}_{S^{-1}}  &\leq  \beta_{k+1}\Big[ \norm{y^\star - \dot{y}}_S  + \big(\norm{y^\star - \dot{y}}_S^2 + 2 \beta_{k+1}^{-1} (F_{\beta_{k+1}}(\bar x^{k+1}) - F^{\star})\big)^{1/2} \Big]. 
\end{align*}
The first inequality is now proved. By Proposition~\ref{t1},
$F_{\beta_{k+1}}(\bar x^{k+1}) - F^{\star} \leq \frac{B_1}{2(k+1)} \norm{x^\star - \tilde x^0}^2$ and by Lemma~\ref{lem:choose_parameters_asgard}, 
$\beta_{k+1} \leq \frac{\beta_0}{k+2}\leq\frac{\beta_0}{k+1}$. Hence we get the second inequality. For the last inequality,
\begin{multline}
\beta_{k+1}\Big[ \norm{y^\star - \dot{y}}_S  + \big(\norm{y^\star - \dot{y}}_S^2 + 2 \beta_{k+1}^{-1} (F_{\beta_{k+1}}(\bar x^{k+1}) - F^{\star})^{1/2} \Big] 
\\= 
\beta_{k+1}\norm{y^\star - \dot{y}}_S  + \big(\beta_{k+1}^2\norm{y^\star - \dot{y}}_S^2 + 2 \beta_{k+1} (F_{\beta_{k}}(\bar x^{k+1}) - F^{\star})\big)^{1/2} 
\end{multline}
so we just need to use again the inequalities $F_{\beta_{k+1}}(\bar x^{k+1}) - F^{\star} \leq \frac{B_1}{2(k+1)} \norm{x^\star - \tilde x^0}^2$ and 
$\beta_{k+1} \leq \frac{\beta_0}{k+1}$ to conclude.

\ref{t3:ii} For the second case, since $h$ is $D_{\YY}$-Lipschitz continuous, it follows from \cite[Corollary~17.19]{BC11_B} that 
$\dom{h^*}\subset B[0,D_{\YY}]$, the ball centered at 0 and with radius $D_{\YY}$. 
Therefore,
\begin{equation}
\begin{aligned}
(\forall k\in\NN)\quad 
h(M\bar x^{k+1}) 
 &= \sup_{y\in\dom{h^*}}\;\big\{\Pair{M\bar x^{k+1}}{y} - h^*(y)\}\\
 &\leq \max_{y\in B[0, D_{\YY}]}\;\big\{\Pair{M\bar x^{k+1}}{y} - h^*(y)\}\\
&\leq \max_{y\in B[0,D_{\YY}]}\;\big\{\Pair{M\bar x^{k+1}}{y} - h^*(y) -\frac{\beta_{k+1}}{2} \norm{y -\dot y}_S^2\big\} 
+\frac{\beta_{k+1}}{2}\max_{y\in B[0, D_{\YY}]}\;\norm{y - \dot y}_S^2\\
&\leq h_{\beta_{k+1}}(M\bar x^{k+1};\dot y) +\frac{\beta_{k+1}}{2}\max_{y\in B[0, D_{\YY}]}\;\|y-\dot y\|_S^2\\
&\leq h_{\beta_{k+1}}(M\bar x^{k+1};\dot y) +\beta_{k+1}\bigg[\max_{y\in B[0,D_{\YY}]}\;\|y\|_S^2 + \|\dot y\|_S^2\bigg] \\
&\leq h_{\beta_{k+1}}(M\bar x^{k+1};\dot y) +\beta_{k+1}\big[D_{\YY}^2 + \|\dot y\|_S^2\big].
\end{aligned}
\end{equation}
The conclusion now follows from Proposition~\ref{t1}.
\end{proof}

Finally, we extend the above approach for the following multivariate minimization problem.
\begin{problem}
\label{pr2}
Let $m$ be a strictly positive integer, let $\HH$ and $(\GG)_{1\leq i\leq m}$ be real Hilbert spaces, let $(\forall i\in\{1,\ldots,m\})$ $M_i\colon\HH\to\GG_i$ be bounded linear operators, 
and let  $f\colon\HH\to\RR$, $g\colon\HH\to\RX$ and $h_i\colon\GG_i\to\RX$ be proper, closed lower semi-continuous convex functions where $f$ is moreover assumed to have $L_f$-Lipschitz gradient. Consider the following problem
\begin{equation}
F^\star = \inf_{x\in\HH}\;f(x)+g(x)+\sum_{i=1}^mh_i(M_ix)
\end{equation}
and suppose that its set of minimizers is non-empty.
\end{problem}
For each $i\in\{1,\ldots, m\}$, let us choose $S_i$ 
to be a positive definite linear operator and $\dot y_i\in\GG_i$. 
Define
\begin{equation}
(\forall \beta\in\RPP)\quad h_{\beta;i}(\cdot; \dot y_i)\colon y_i\mapsto\max\limits_{\bar y_i\in\GG_i}\{\Pair{y_i}{\bar y_i} - h_i^*(\bar y_i) -\frac \beta 2 \norm{\bar y_i -\dot y_i}_{S_i}^2 \}.
\end{equation}
The following algorithm is an extension of Algorithm~\ref{eq2} to solve Problem~\ref{pr2}. 
The other variants are similar.

\begin{algorithm}
\begin{algorithmic}[1]
\STATE{Inputs: $\|M\|_{S^{-1}}^2=\sum_{i=1}^m\|M_i\|_{S_i^{-1}}^2$, $\tau_0=1$, $\beta_0 >0$, $\bar x^0\in\HH$, $\tilde x^0\in\HH$.}
\FOR{$k = 0, 1, \ldots$}
\STATE{$\hat x^k = (1-\tau_k) \bar x^k + \tau_k \tilde x^k$}
\STATE{$\beta_{k+1} = \frac{\beta_k}{1+\tau_k}$ and $B_{k+1}=L_f + \frac{\|M\|_{S^{-1}}^2}{\beta_{k+1}}$}
\FOR{$i = 1,\ldots, m$}
\STATE{$y_{k;i}^* =  \argmax{y_i\in\GG_i}{\Pair{M_i\hat x^k}{y_i} - h_i^*(y_i) -\frac{\beta_{k+1}}{2} \norm{y_i -\dot y_i}_{S_i}^2}$}
\ENDFOR
\STATE{$v^k =  \sum_{i=1}^mM_i^*y_{k;i}^*$}
\STATE{$\tilde x^{k+1} = \prox_{\frac{1}{\tau_kB_{k+1}}g}\big(\tilde x^k - \frac{1}{\tau_kB_{k+1}}(\nabla f(\hat x^k) + v^k)\big)$}
 \STATE{$\bar x^{k+1} = \hat x^k + \tau_k(  \tilde x^{k+1} - \tilde x^k) = (1-\tau_k) \bar x^k + \tau_k \tilde x^{k+1}$}
\STATE{Find the unique positive $\tau_{k+1}$ such that}
 \[
 \frac{B_{k+1} - L_f}{B_{k+1}}\tau_{k+1}^3 + \tau_{k+1}^2 +  \tau_{k}^2\tau_{k+1} -  \tau_{k}^2 = 0
 \]
\ENDFOR
\RETURN $\bar x^{k+1}$
\end{algorithmic}
\caption{Parallel ASGARD}
\label{eq3}
\end{algorithm}

\begin{proposition}
\label{t2}
Consider the setting of Problem~\ref{pr2}, let $(\bar x^k)_{k\in\NN}$ be the sequence 
generated by Algorithm~\ref{eq3} and define 
\begin{equation}
(\forall k\in\NN)\quad F_{\beta_k}\colon x\mapsto f(x)+ g(x) + \sum_{i=1}^mh_{\beta_k;i}(M_ix;\dot y_i).
\end{equation} 
Then for  any solution $x^\star$ to Problem~\ref{pr2}, we have
\begin{equation}
(\forall k\in\NN)\quad F_{\beta_{k+1}}(\bar x^{k+1}) - F^\star \leq\dfrac{B_1}{2(k+1)}\|x^\star-\tilde x^0\|^2.
\end{equation}
\end{proposition}
\begin{proof}
Let $\GG = \bigoplus_{i=1}^m\GG_i$ be the Hilbert direct sum of $(\GG_i)_{1\leq i\leq m}$. 
For $y =(y_i)_{1\leq i\leq m}\in\GG$ and $\hat y = (\hat y_i)_{1\leq i\leq m}\in\GG$, we define
$\norm{y}=\sqrt{\sum_{i=1}^m\|y_i\|^2}$ and $\Pair{y}{\hat y}=\sum_{i=1}^{m}\Pair{y_i}{\hat y_i}$. 
Set
\begin{equation}
\begin{cases}
M\colon\HH\to\GG\colon x\mapsto (M_ix)_{1\leq i\leq m},\\
h\colon\GG\to\RX\colon (y_i)_{1\leq i\leq m}\mapsto\sum_{i=1}^mh_i(y_i),\\
S\colon\GG\to\RX\colon (y_i)_{1\leq i\leq m}\mapsto\sum_{i=1}^m S_i y_i.
\end{cases}
\end{equation}
Then Problem~\ref{pr2} reduces to Problem~\ref{p1}. 
It is easy to see that $S$ is a positive definite operator on $\GG$ and it induce the norm
\begin{equation}
(\forall y =(y_i)_{1\leq i\leq m}\in\GG)\quad \|y\|_S = \sqrt{\sum_{i=1}^m\|y_i\|_{S_i}^2}. 
\end{equation}
Moreover, for any $\beta\in\RPP$ and any $y\in\GG$, we have
\begin{equation}
\begin{aligned}
h_{\beta}(y; \dot y)
& =\max\limits_{\bar y\in\GG}\big\{\Pair{y}{\bar y} - h^*(\bar y) - \frac{\beta}{2} \norm{\bar y - \dot y}_S^2\big\}\\
&=\max\limits_{\bar y\in\GG}\bigg\{\sum_{i=1}^m\Pair{y_i}{\bar y_i} - \sum_{i=1}^mh_i^*(\bar y_i) -\frac \beta 2 \sum_{i=1}^m \norm{\bar y_i-\dot y_i}_{S_i}^2\bigg\}\\
&=\sum_{i=1}^m\max\limits_{\bar y_i\in\GG_i}\big\{\Pair{y_i}{\bar y_i} - h_i^*(\bar y_i) -\frac \beta 2 \norm{\bar y_i-\dot y_i}_{S_i}^2 \big\}\\
& = \sum_{i=1}^mh_{\beta;i}(y_i; \dot y_i),
\end{aligned}
\end{equation}
where $\dot y = (\dot y_1, \ldots, \dot y_m)$. The assertion then follows from Proposition~\ref{t1}.
\end{proof}

\subsection{Restarting}
It is possible to restart our variants of ASGARD using a fixed iteration restarting strategy, 
i.e., restart every $q$ iterations, as follows: 
\begin{equation}
\begin{cases}
\hat x^{k+1} \leftarrow \bar x^{k+1},\\
\dot y \leftarrow y^*_{\beta_{k+1}}(M\hat x^k; \dot y),\\
\beta_{k+1}\leftarrow \beta_0,\\
\tau_{k+1}\leftarrow 1.
\end{cases}
\end{equation}
The performance of different variants of ASGARD with restarting will be illustrated in the next section.

\section{Numerical experiments}
\label{sct:num}
\label{experiments}
\subsection{Sparse and TV regularized least squares} 
We consider the following regularized least squares problem
 \[
 \min_{x\in\RR^{100}} \frac 12 \norm{A x -b}^2 + \norm{x}_1 + \norm{D x}_1
 \]
 where $A$ is a randomly generated matrix of size 50 $\times$ 100 (Gaussian distribution, covariance 
 $\Sigma_{i,j} = \rho^{|i-j|}$ with $\rho = 0.95$), $b$ is randomly generated   
($b_i$ iid, with uniform distribution on $[1, 2]$) and $D$ is the explicit 1D discrete gradient operator. 
This problem is a special case of \eqref{prob1} with 
\begin{equation}
f(x) = \frac{1}{2}\|Ax- b\|^2 ,\quad g(x) =  \|x\|_1, \quad h(x) = \|x\|_1 \quad\text{and} \quad M = D.
\end{equation}
In this case,
\begin{equation}
\prox_{\gamma g}(x) = \operatorname{soft}_{[-\gamma,\gamma]}(x)\quad\text{and}\quad
\prox_{\gamma h^*}(x) = x - \gamma\operatorname{soft}_{[-\gamma^{-1}, \gamma^{-1}]}(\gamma^{-1}x),
\end{equation}
where
\begin{equation}
\operatorname{soft}_{[-t,t]}(x) = \operatorname{sign}(x)\otimes\max\{|x| - t, 0\},
\end{equation}
here $\otimes$ denotes component-wise multiplication.

When the plot has dash-dotted line, we consider the constraint $z = Dx$ and the augmented primal variable $(x, z)$, otherwise, we directly split with $h = \norm{\cdot}_1$. For ASGARD with restart, we restart the momentum in the algorithm every 100 iterations; vu-condat is Vu-Condat's algorithm and ladmm is the linearized ADMM method. For each algorithm we plot the difference between the current function value and the best function value encountered in the experiment (Figure~\ref{fig:l11DTV}).
\begin{figure}[!h]
\begin{center}
\includegraphics[width = 11 cm, trim = 0em 1.5em 0em 1.5em, clip]{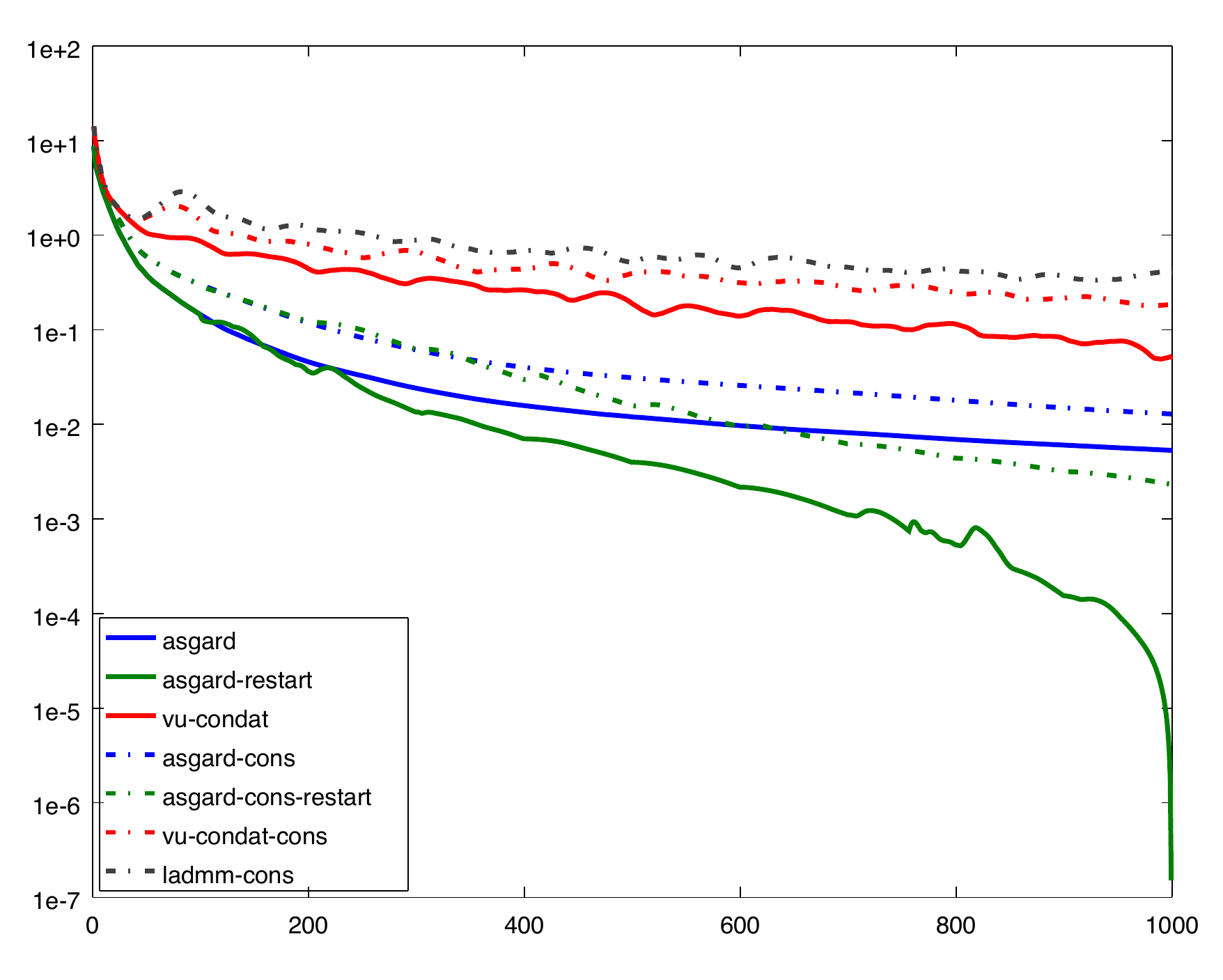}
\end{center}
 \caption{Behavior of various algorithms for the synthetic sparse and TV regularized least squares problem:  we plot the difference between the current function value and the best function value encountered vs iterations.}
 \label{fig:l11DTV}
\end{figure}

We also considered a medium-scale sparse and TV regularized problem on functional MRI data~\cite{tom2007neural}.
For given regularization parameters $\alpha>0$ and $r \in [0,1]$, we would like to solve the following  regression problem with regularization given by the sum of Total Variation (TV) and the $\ell_1$ norm:
\[
\min_{x \in \mathbb{R}^n} \frac{1}{2}\norm{Ax - b}_2^2 + \alpha\big( r \norm{x}_1 + (1-r) \norm{M x}_{2,1}\big).
\]
The problem takes place on a 3D image of the brains of size $40 \times 48 \times 34$. 
The optimization variable $x$ is a real vector with one entry in each voxel, that is $n$ = 65280. 
Matrix $M$ is the discretized 3D gradient. This is a sparse matrix of size 195840 $\times$ 65280
with 2 nonzero elements in each row.
The matrix $A \in \mathbb{R}^{768 \times 65280}$ and the vector $b \in \mathbb{R}^{768}$ correspond to 768 labeled experiments where each line of $A$ 
gathers brains activity for the corresponding experiment. Parameter $r$ tunes the tradeoff between the two regularization terms.
We chose $r = 0.1$ and $\alpha = 0.1$.

In this scenario, we set the objective as $f(x) = \frac{1}{2}\norm{Ax - b}_2^2 $, $g(x) = \alpha  r \norm{x}_1 $
and $h(y) = \alpha (1-r) \norm{y}_{2,1}$.
On Figure~\ref{ffmri:l13DTV}, we compared our algorithms against FISTA~\cite{BT09} 
with an inexact resolution of the proximal operator of TV, FISTA restarted every 30 iterations, and V\~{u}-Condat's algorithm~\cite{Vu13,condat13}. We can see that on this problem, 
ASGARD outperforms V\~u-Condat's algorithm but not FISTA. After careful inspection, we realize that 
ASGARD (and also V\~{u}-Condat's algorithm) spends too much time computing gradients of $f$ while FISTA 
spends much of its time to compute the proximity operator of $g$ (Figures~\ref{ffmri:l13DTVa}~and~\ref{ffmri:l13DTVb}). 
Our framework allows us to consider useful variants in this setting.
For instance, the use of old gradients makes ASGARD much faster in this problem in time. 
We can see on Figure~\ref{ffmri:l13DTV} that \texttt{ASGARD\_Old\_Gradients} 
outperforms FISTA. Moreover, combined with a restart every 400 iterations, we obtain the best performance among the algorithms we test.

\begin{figure}[!h]
\begin{center}
\includegraphics[width = 16 cm, trim={0em, 0em, 0em, 0ex}, clip]{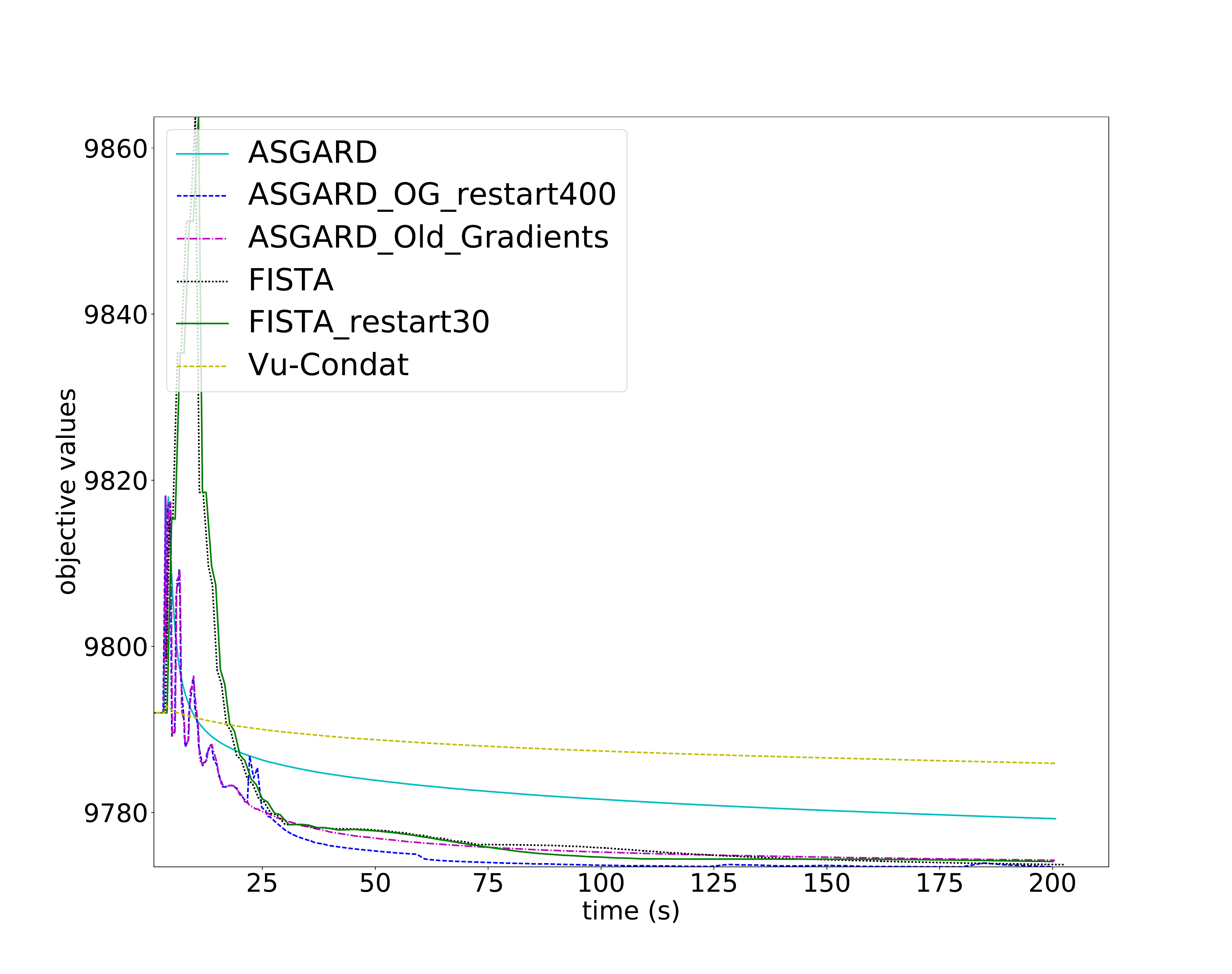}
\end{center}
 \caption{Comparison of  various algorithms for the functional MRI problem: function value against computational time.}
 \label{ffmri:l13DTV}
\end{figure}

\begin{figure}[!h]
\begin{center}
\includegraphics[width = 13 cm, trim={0em, 0em, 0em, 0ex}, clip]{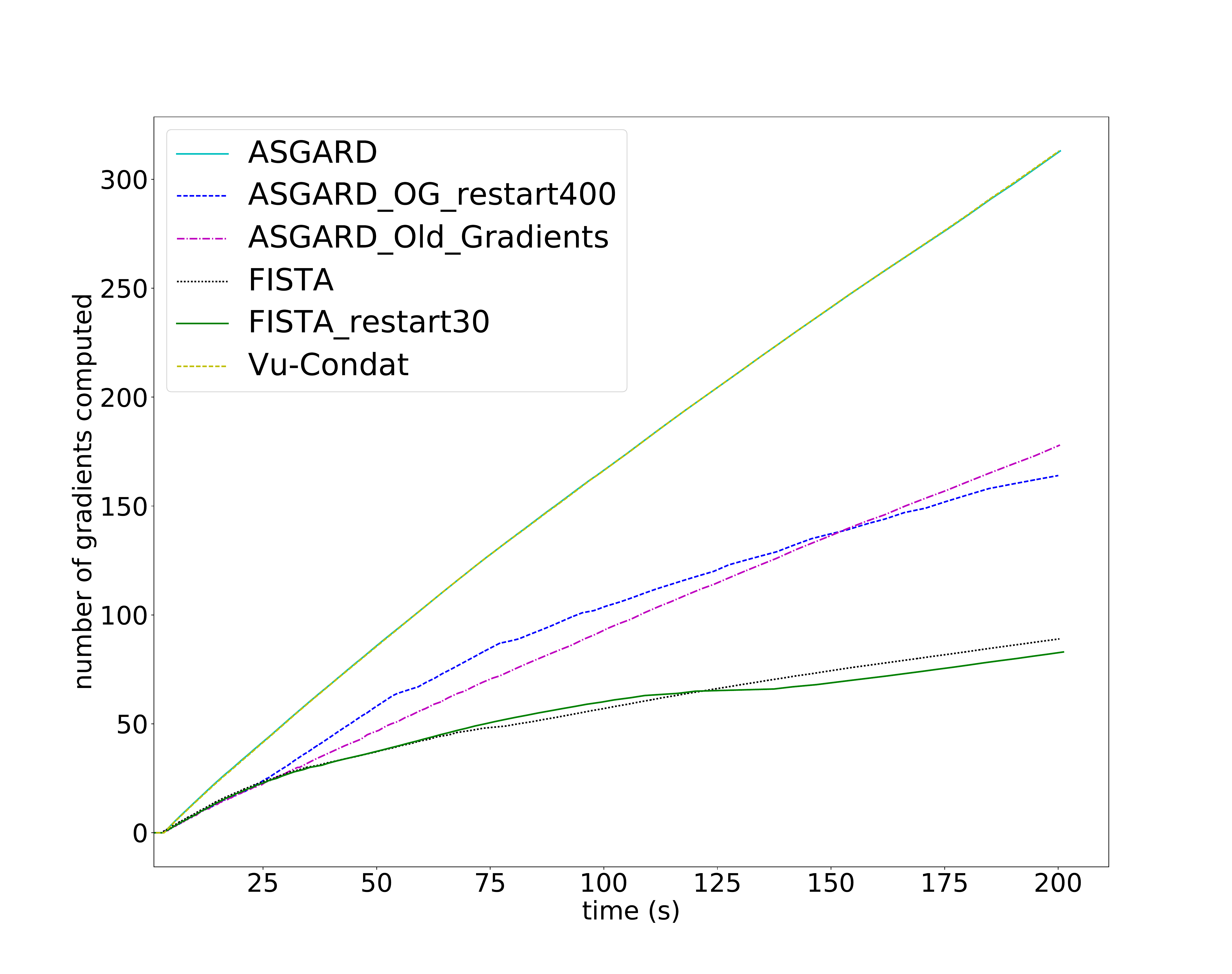}
\end{center}
 \caption{Comparison of  various algorithms for the functional MRI problem: number of gradients evaluations against computational time.}
 \label{ffmri:l13DTVa}
\end{figure}
\begin{figure}[!h]
\begin{center}
\includegraphics[width = 13 cm, trim={0em, 0em, 0em, 0ex}, clip]{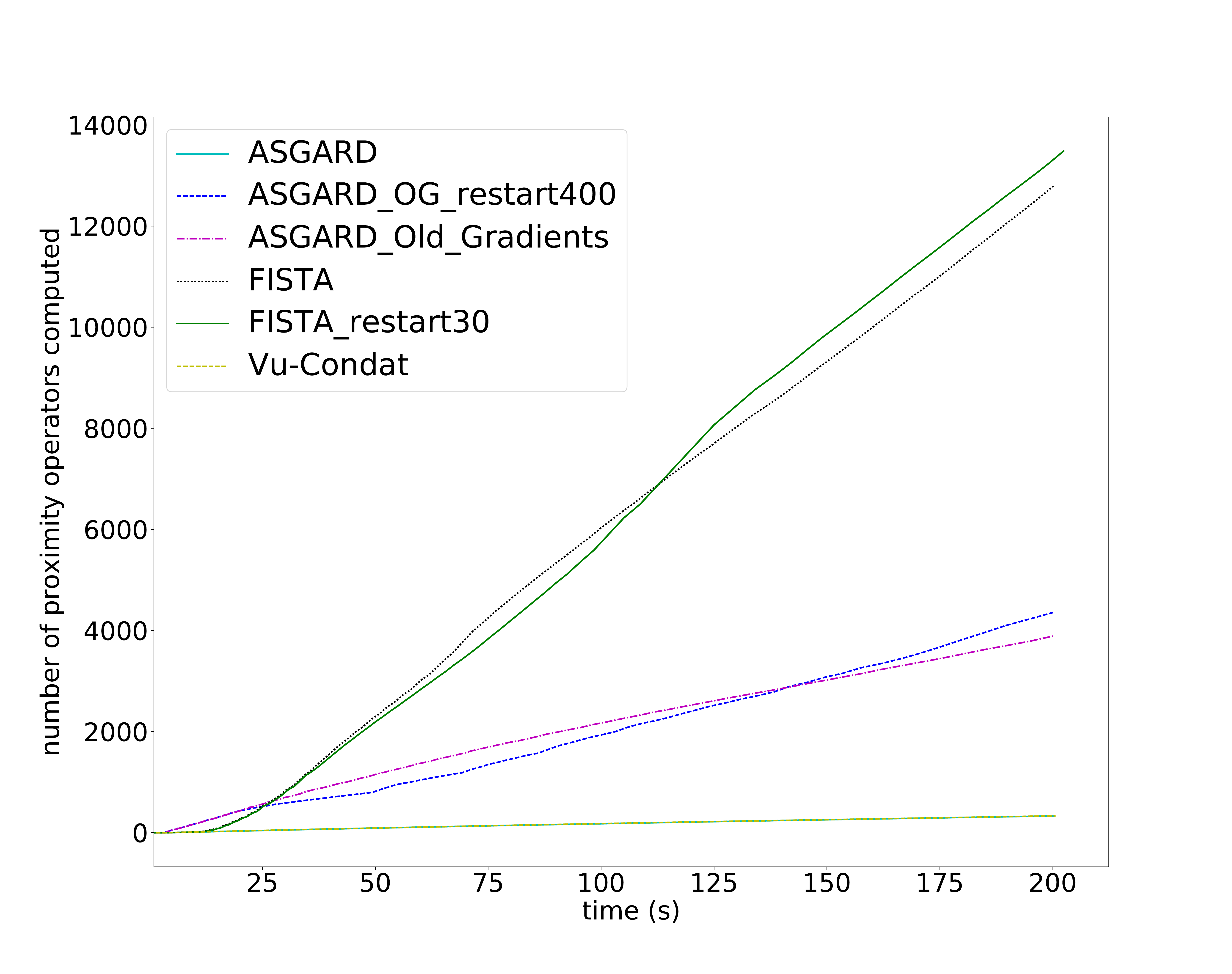}
\end{center}
 \caption{Comparison of  various algorithms for the functional MRI problem: number of prox evaluations against computational time.}
 \label{ffmri:l13DTVb}
\end{figure}

\subsection{Quantum properties prediction}

In materials science, quantum properties such as energy requires expensive calculations based on the density functional theory (DFT). 
Machine learning has been recently used to predict such properties for new molecules based on dataset derived by DFT.  
Let us represent the dataset by $\{(r_i,  p_i)\}_{i=1}^N$ where $r_i\in\mathbb{R}^n$ 
is Coulomb matrix representation \cite{Rupp:2015aa} of  $i$-th molecule and $p_i\in\mathbb{R}$ is its properties.  
In this experiment, the Laplacian kernel, i.e., $K(r,r') = \exp(-\|r -r'\|_1/\sigma)$ 
with $\|\cdot\|_1$ is $\ell_1-$norm of $\mathbb{R}^n$, is used to measure the dissimilarity between molecules.
A quantum property of a molecule with representation $r$ is assumed to have the following form
\begin{equation}
e(r) = \sum_{i=1}^Nx_iK(r,r_i).
\end{equation}
The regression coefficients $x = (x_1, \ldots, x_N)^T$ are obtained 
by solving the following elastic net regularized minimization problem
\begin{equation}
\label{elas1}
\minimize{x\in\mathbb{R}^N}{\|Kx - p\|_1 + \frac{\lambda}{2}x^TKx 
+ (1-\lambda)\|x\|_1},
\end{equation}
here $p = (p_1, \ldots, p_N)^T$ and $K_{ij} = K(r_i,r_j)$. 
Note that \eqref{elas1} is a particular case of \eqref{prob1} with 
\begin{equation}
f(x) = \frac{\lambda}{2}x^TKx ,\quad 
g(x) =  (1- \lambda)\|x\|_1, \quad h(y) = \|y - p\|_1 
\quad\text{and}\quad  M \colon x\mapsto Kx.
\end{equation}
In this case,
\begin{equation}
\prox_{\gamma g}(x) = \operatorname{soft}_{[-(1- \lambda)\gamma,(1- \lambda)\gamma]}(x)
\quad\text{and}\quad
\prox_{\gamma h^*}(x) = x - \gamma\big(p 
+ \operatorname{soft}_{[-\gamma^{-1}, \gamma^{-1}]}(\gamma^{-1}x - p)\big).
\end{equation}
In Figure~\ref{fig:c1}, we compare the behavior of different versions of our ASGARD 
with Vu-Condat's algorithm \cite{condat13,Vu13} and Combettes-Pesquet's algorithm \cite{CP12} 
on the dataset of $7211$ molecules in \cite{Rupp:2015aa} in which $50\%$ molecules 
are used to train.\cite{CP12}.
\begin{figure}[!h]
\begin{center}
\includegraphics[width = 13 cm]{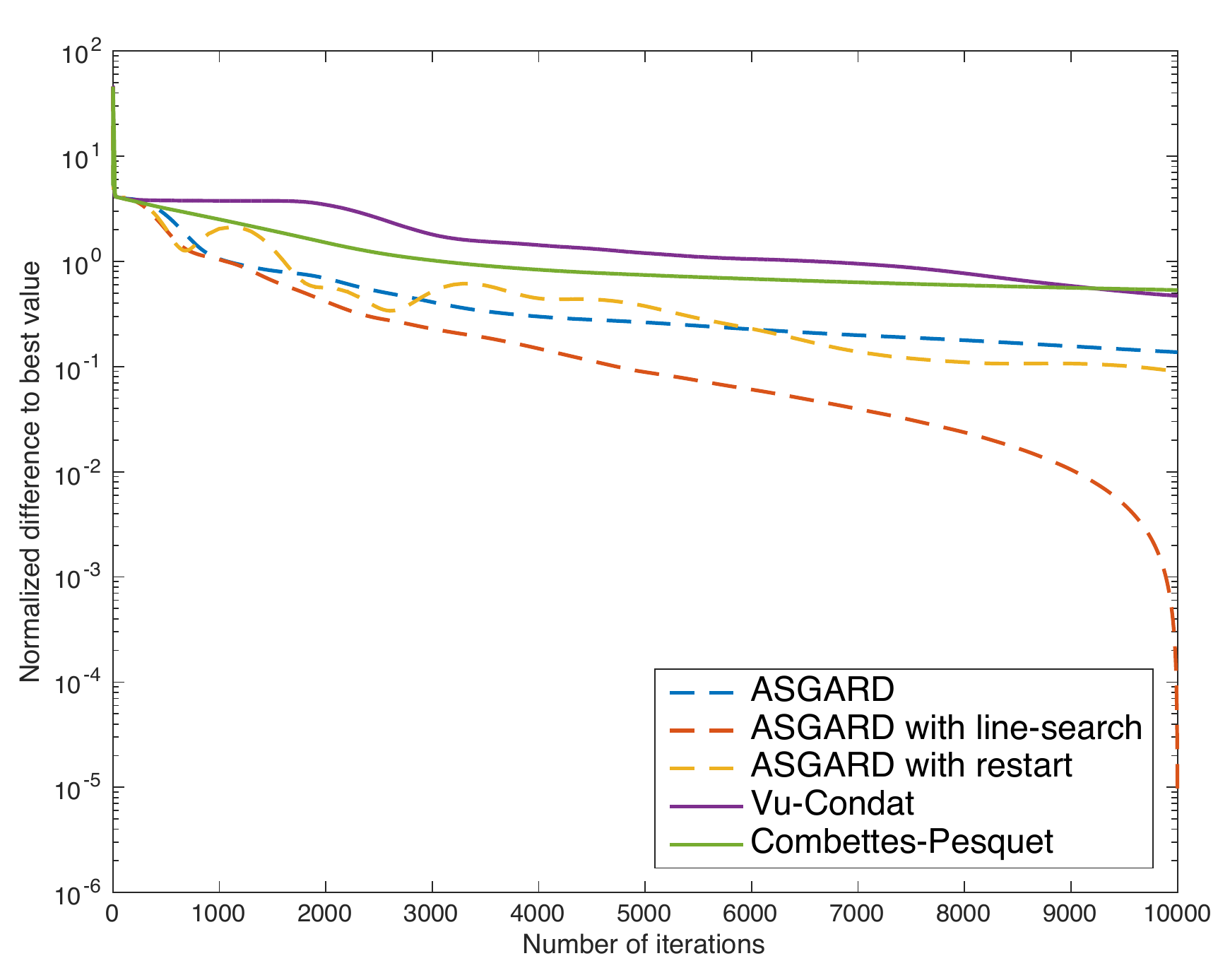}
\end{center}
 \caption{Comparison with existing algorithms ($\sigma = 4000$ and $\lambda = 0.001$).}
 \label{fig:c1}
\end{figure}

\section{Conclusion}
In this paper, we build, based on the homotopy-based smoothing and acceleration technique of [ASGARD], a new method to solve a large class of generic convex optimization problems where the objective function is split into a sum of one smooth term and two non-smooth terms, one of which is combined with a linear operator. The variants of our method with line-search and old gradients benefits from the local smoothness of nonsmooth function and can avoid computing the whole gradient of the smooth function. In contrast to the existing methods in the literature, our method also features rigorous convergence guarantees. Numerical experiments with real-world problems illustrate the superiority of our method vs. the other state-of-the-art algorithms.

\section*{Acknowledgments.}
The work of V. Cevher and Q. V. Nguyen is supported by the NCCR MARVEL, funded by the Swiss National Science Foundation.
The work of O. Fercoq is supported by a public grant as part of the
Investissement d'avenir project, reference ANR-11-LABX-0056-LMH,
LabEx LMH and PGMO.

\bibliographystyle{plain}
\bibliography{bibref}
\end{document}